%% file: RigidityConvergence_final_A.tex
\documentclass[reqno]{amsart}

\usepackage[T1]{fontenc}
\usepackage{graphicx}
\usepackage{mathptmx}      
\usepackage{pinlabel}
\usepackage{amsmath}
\usepackage{amsfonts}
\usepackage{amssymb}
\usepackage{enumerate}
\usepackage{expdlist}

\usepackage{subfig}
\usepackage{bbold}
\usepackage[usenames]{color}
\usepackage{dsfont}
\usepackage{mathrsfs}
\usepackage{url}
\usepackage{tikz}
\usetikzlibrary{decorations.pathreplacing}
\usetikzlibrary{calc}
\usetikzlibrary{shapes.geometric}
\usepackage{empheq}
\usepackage{tkz-euclide}
\usepackage{pgfplots}
\usepgfplotslibrary{ternary}

\newcommand{\R}{\mathbb R}
\newcommand{\C}{\mathbb C}
\newcommand{\N}{\mathbb N}

\newcommand{\Sp}{\mathbb S}
\newcommand{\D}{\mathbb D}

\newcommand{\eps}{\varepsilon}

\renewcommand{\Im}{\mathrm{Im}}
\newcommand{\Li}{\;\mathrm{Li}_2}

\newcommand{\m}{\textsc{Mod}}
\newcommand{\vel}{\textsc{Vel}}
\newcommand{\ci}{\mathbb{c}}
\newcommand{\B}{\mathbb{B}}

\theoremstyle{plain}
\newtheorem{theorem}{Theorem}[section]
\newtheorem{proposition}[theorem]{Proposition}
\newtheorem{corollary}[theorem]{Corollary}
\newtheorem{lemma}[theorem]{Lemma}
\theoremstyle{definition}
\newtheorem{remark}[theorem]{Remark}

\theoremstyle{definition}
\newtheorem{definition}[theorem]{Definition}

\title{On rigidity and convergence of circle patterns}

\author{Ulrike B\"ucking
}

\address{Ulrike B\"ucking, 
              Technische Universit\"at Berlin, 
Institut f\"ur Mathematik, 
Stra\ss{}e des 17.\ Juni 136, 10623 Berlin, Germany, 
}
\email{buecking@math.tu-berlin.de}

\begin{document}


\begin{abstract}
\small
Two planar embedded circle patterns with the same combinatorics and 
the same intersection angles can be considered to define a discrete conformal 
map. We show that two locally finite circle patterns covering the 
unit disc are related by a hyperbolic isometry. Furthermore, we prove an 
analogous rigidity statement for the complex plane if all exterior 
intersection angles of neighboring circles are 
uniformly bounded away from $0$.

Finally, we study a sequence of two circle 
patterns with the same combinatorics each of which approximates a given simply  
connected domain. Assume 
that all kites are convex and all angles in the kites are uniformly bounded and 
the radii of one circle pattern converge to $0$. Then a subsequence of the 
corresponding discrete conformal maps converges to a Riemann map between the 
given domains.
%
\end{abstract}

\maketitle

\section{Introduction}

Holomorphic mappings of a domain in the complex plane $\C$ with non-vanishing 
derivative -- also called {\em conformal maps} -- build an important class of 
functions in complex analysis which has 
been investigated for many decades. The more recent idea of studying `discrete 
analogues' led to various approaches to define {\em discrete conformal maps}. 
In this article, we are concerned with pairs of circle patterns which can be 
considered as discrete conformal maps. We show that under suitably defined 
conditions these maps converge to smooth conformal maps. Furthermore, we prove 
rigidity of infinite patterns filling the complex plane or the unit disc.

In particular, assume we are given a cell decomposition 
$\mathscr K$ of (a part of) $\C$ into kites. More precisely, $\mathscr K$ is
a strongly regular cell decomposition whose 2-cells (faces) are 
embedded quadrilaterals (in fact kites), that is, there are exactly four edges 
incident to each face, with counterclockwise orientation.
Furthermore, we assume that the 1-skeleton of $\mathscr K$ is a bipartite 
graph whose edges are straight line segments and that the vertices are colored 
white and black as in Figure~\ref{figbquad}~(left). 
For every white vertex we suppose that the lengths of the incident egdes agree 
as in Figure~\ref{figbquad}~(right) and add a circle 
with center at this vertex which passes through the incident black vertices of 
$\mathscr K$. So white vertices correspond to centers of 
circles and black vertices correspond to intersection points. An example is 
illustrated in Figure~\ref{figbquad}~(right). This leads to an {\em embedded 
planar circle pattern} $\mathscr C$. Obviously, the combinatorics of this 
circle pattern is described by the combinatorics of $\mathscr K$. Furthermore, 
we can read off the (exterior) {\em intersection angles} $\alpha$ between the 
circles, as illustrated in Figure~\ref{figKite}~(left). We associate this 
intersection angle to the kites of $\mathscr K$.

\begin{figure}[tb]
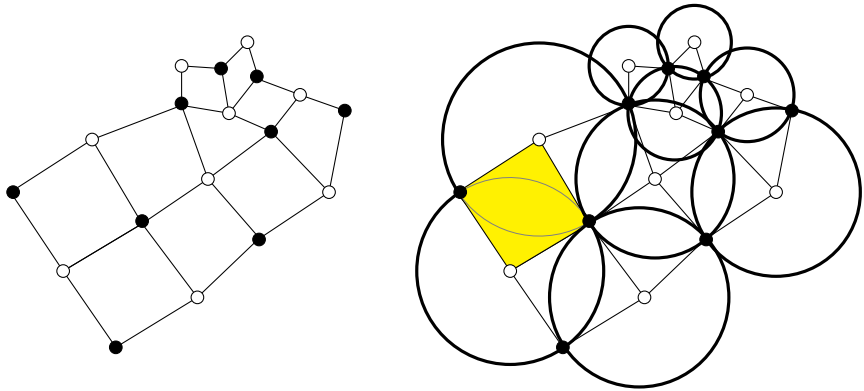


\ifx\XFigwidth\undefined\dimen1=0pt\else\dimen1\XFigwidth\fi
\divide\dimen1 by 7643
\ifx\XFigheight\undefined\dimen3=0pt\else\dimen3\XFigheight\fi
\divide\dimen3 by 6587
\ifdim\dimen1=0pt\ifdim\dimen3=0pt\dimen1=4143sp\dimen3\dimen1
  \else\dimen1\dimen3\fi\else\ifdim\dimen3=0pt\dimen3\dimen1\fi\fi

\tikzpicture[x=+\dimen1, y=+\dimen3,scale=0.35]
\clip(847,-7863) rectangle (8490,-1276);
\tikzset{inner sep=+0pt, outer sep=+0pt}
\draw (3375,-7155)--(4770,-6300)--(3825,-5040)--(2475,-5850)--cycle;
\draw (2970,-3600)--(1575,-4500)--(2475,-5850)--(3825,-5040)--cycle;
\draw (2970,-3600)--(4500,-3015)--(4950,-4275);
\draw (4500,-3015)--(4500,-2340)--(5175,-2385);
\draw (5805,-5310)--(7020,-4500)--(6030,-3465)--(4950,-4275);
\draw (3825,-5040)--(4950,-4275)--(5850,-5310);
\draw (4770,-6300)--(5805,-5310);
\draw (6525,-2835)--(7290,-3105)--(7020,-4500);
\draw (6525,-2835)--(6030,-3465);
\draw (5805,-2565)--(6525,-2835);
\draw (5310,-3150)--(6030,-3465);
\draw (5310,-3150)--(4500,-3015);
\draw (5310,-3150)--(5175,-2340);
\draw (5625,-1935)--(5175,-2385);
\draw (5625,-1935)--(5805,-2520);
\draw (5805,-2520)--(5310,-3150);
\filldraw  (4500,-3015) circle [radius=+64];
\draw[fill=white]  (2475,-5850) circle [radius=+108];
\draw[fill=white]  (4770,-6300) circle [radius=+108];
\filldraw  (3375,-7155) circle [radius=+108];
\filldraw  (3825,-4995) circle [radius=+108];
\filldraw  (1620,-4500) circle [radius=+108];
\draw[fill=white]  (2970,-3600) circle [radius=+108];
\draw[fill=white]  (4950,-4275) circle [radius=+108];
\filldraw  (5824,-5310) circle [radius=+108];
\draw[fill=white]  (7020,-4500) circle [radius=+108];
\filldraw  (7290,-3105) circle [radius=+108];
\draw[fill=white]  (6525,-2835) circle [radius=+108];
\draw[fill=white]  (5625,-1935) circle [radius=+108];
\draw[fill=white]  (4500,-2340) circle [radius=+108];
\draw[fill=white]  (5310,-3150) circle [radius=+108];
\filldraw  (6030,-3465) circle [radius=+108];
\filldraw  (5786,-2520) circle [radius=+108];
\filldraw  (5175,-2385) circle [radius=+108];
\filldraw  (4498,-2982) circle [radius=+108];
\endtikzpicture%
\tikzpicture[x=+\dimen1, y=+\dimen3,scale=0.35]
\clip(847,-7863) rectangle (8490,-1276);
\tikzset{inner sep=+0pt, outer sep=+0pt}
\draw (3375,-7155)--(4770,-6300)--(3825,-5040)--(2475,-5850)--cycle;
\draw (2970,-3600)--(4500,-3015)--(4950,-4275);
\draw (4500,-3015)--(4500,-2340)--(5175,-2385);
\draw (5805,-5310)--(7020,-4500)--(6030,-3465)--(4950,-4275);
\draw (3825,-5040)--(4950,-4275)--(5850,-5310);
\draw (4770,-6300)--(5805,-5310);
\draw (6525,-2835)--(7290,-3105)--(7020,-4500);
\draw (6525,-2835)--(6030,-3465);
\draw (5805,-2565)--(6525,-2835);
\draw (5310,-3150)--(6030,-3465);
\draw (5310,-3150)--(4500,-3015);
\draw (5310,-3150)--(5175,-2340);
\draw (5625,-1935)--(5175,-2385);
\draw (5625,-1935)--(5805,-2520);
\draw (5805,-2520)--(5310,-3150);

\draw[very thick]  (2475,-5850) circle [radius=+1598];
\draw[very thick]  (2970,-3600) circle [radius=+1652];
\draw[fill=yellow] 
(2970,-3600)--(1575,-4500)--(2475,-5850)--(3825,-5040)--cycle;
\draw[ultra thin, color=black!50] (3828,-5011) arc[start 
angle=-58.7, end angle=-143.5, radius=+1652];
\draw[ultra thin, color=black!50] (3828,-5011) arc[start 
angle=+31.6, end angle=+119.5, radius=+1598];

\filldraw  (4500,-3015) circle [radius=+64];
\draw[very thick]  (5287,-3150) circle [radius=+799];
\draw[fill=white]  (2475,-5850) circle [radius=+108];
\draw[fill=white]  (4770,-6300) circle [radius=+108];
\filldraw  (3375,-7155) circle [radius=+108];
\draw[very thick]  (4680,-6300) circle [radius=+1533];
\filldraw  (3825,-4995) circle [radius=+108];
\filldraw  (1620,-4500) circle [radius=+108];
\draw[fill=white]  (4950,-4275) circle [radius=+108];
\draw[very thick]  (4500,-2340) circle [radius=+676];
\filldraw  (5824,-5310) circle [radius=+108];
\draw[fill=white]  (7020,-4500) circle [radius=+108];
\filldraw  (7290,-3105) circle [radius=+108];
\draw[very thick]  (7020,-4500) circle [radius=+1440];
\draw[very thick]  (6525,-2835) circle [radius=+801];
\draw[very thick]  (4950,-4275) circle [radius=+1351];
\draw[fill=white]  (2970,-3600) circle [radius=+108];
\draw[fill=white]  (6525,-2835) circle [radius=+108];
\draw[very thick]  (5625,-1935) circle [radius=+630];
\draw[fill=white]  (5625,-1935) circle [radius=+108];
\draw[fill=white]  (4500,-2340) circle [radius=+108];
\draw[fill=white]  (5310,-3150) circle [radius=+108];
\filldraw  (6030,-3465) circle [radius=+108];
\filldraw  (5786,-2520) circle [radius=+108];
\filldraw  (5175,-2385) circle [radius=+108];
\filldraw  (4498,-2982) circle [radius=+108];
\endtikzpicture%
\caption{{\it Left:} An example of a b-quad-graph $\mathscr D$ (black edges, 
quadrilateral faces and bicolored vertices); {\it right:} A corresponding circle 
pattern $\mathscr C$ for 
$\mathscr D$ with indicated kite pattern $\mathscr K$ isomorphic to $\mathscr 
D$, one of the kites is colored}
\label{figbquad}
\end{figure}

As we are interested in {\em different} circle patterns with the same 
combinatorics and the same intersection angles, we describe the combinatorics 
using an abstract strongly regular cell decomposition $\mathscr D$ 
of (a part of) $\C$ which may be realized as a cell decomposition $\mathscr K$ 
with kites.
The set of faces of $\mathscr D$ is denoted by $F({\mathscr D})$. We always 
assume that the 1-skeleton of $\mathscr D$ is a bipartite 
graph and that the vertices are colored white and black as in 
Figure~\ref{figbquad}~(left). Such a cell 
decomposition $\mathscr D$  will be called {\em b-quad-graph}. 
We are interested in planar embedded circle patterns $\mathscr C$ for a given
b-quad-graph $\mathscr D$ which means that the combinatorics of the 
corresponding kite pattern $\mathscr K$ is isomorphic to $\mathscr D$.
In this article, we only consider {\em embedded} circle patterns, that is,
different kites have mutually disjoint interiors.
Therefore, we sometimes identify (the abstract 
cell decomposition) $\mathscr D$ with the kite 
pattern $\mathscr K$ and we will frequently omit the word `embedded'.

To the quadrilateral faces of a b-quad graph $\mathscr D$ we associate a 
function $\alpha:F({\mathscr D})\to (0,\pi)$, called {\em labelling}. 
Given an embedded circle pattern $\mathscr C$ for $\mathscr D$ and the 
corresponding pattern of kites $\mathscr K$, we can read off for every kite
the (exterior) {intersection angle} between the circles, as 
illustrated in Figure~\ref{figKite}~(left),
and compare this to
the value $\alpha(f)$ for the corresponding face $f$ of $\mathscr D$. If all 
these values agree, then $\mathscr C$ is called a {\em circle pattern for 
$\mathscr D$ and $\alpha$}. Note that we only consider {\em planar} circle 
patterns and will therefore omit the notion `planar' in the following.
But planarity of the circle pattern implies that the 
intersection angles at an interior black vertex sum up to $2\pi$. Therefore, we 
demand as a necessary condition for the labelling $\alpha$ that at 
all interior black vertices $v$ we have
\begin{equation}\label{condalpha}
  \sum_{f \text{ incident to } v} \alpha(f) = 2\pi.
\end{equation}
Such a labelling $\alpha$ of the faces will be called {\em admissible}.
\par\medskip

\begin{figure}
\ifx\XFigwidth\undefined\dimen1=0pt\else\dimen1\XFigwidth\fi
\divide\dimen1 by 1821
\ifx\XFigheight\undefined\dimen3=0pt\else\dimen3\XFigheight\fi
\divide\dimen3 by 1176
\ifdim\dimen1=0pt\ifdim\dimen3=0pt\dimen1=4143sp\dimen3\dimen1
  \else\dimen1\dimen3\fi\else\ifdim\dimen3=0pt\dimen3\dimen1\fi\fi
\tikzpicture[x=+\dimen1, y=+\dimen3]
\clip(2672,-2713) rectangle (4493,-1537);
\tikzset{inner sep=+0pt, outer sep=+0pt}
\pgfsetfillcolor{white}
\filldraw  (3093,-2125) circle [radius=+33];
\filldraw  (3904,-2125) circle [radius=+33];
\pgfsetcolor{.}
\filldraw  (3402,-1865) circle [radius=+33];
\filldraw  (3402,-2401) circle [radius=+33];
\draw[very thick]  (3904,-2125) circle [radius=+574];
\draw[very thick]  (3093,-2125) circle [radius=+406];
\pgfsetfillcolor{yellow}
\filldraw (3095,-2126) --(3403,-1861)--(3907,-2126)--(3398,-2391)--cycle;
\draw[ultra thin, color=black!50] (3398,-2393) arc[start angle=+-41.2, end 
angle=+25.5, radius=+406];
\draw[ultra thin, color=black!50] (3398,-2393) arc[start angle=+208.2, end 
angle=+165.5, radius=+574];
\fill[fill=yellow]
(3271,-1979)--(3273,-1980)--(3277,-1982)--(3285,-1986)--(3294,-1990)--(3306,
-1995)
  
--(3319,-1999)--(3334,-2004)--(3351,-2007)--(3370,-2010)--(3392,-2011)--(3417,
-2011)
  
--(3442,-2009)--(3466,-2005)--(3486,-2000)--(3505,-1994)--(3522,-1988)--(3538,
-1982)
  
--(3552,-1976)--(3563,-1970)--(3572,-1966)--(3577,-1963)--(3580,-1962) --(3402,
-1865)--cycle;
\draw (3271,-1979) arc[start angle=+-115.0, end angle=+-57.5,radius=+328];
\coordinate [label={below:$\alpha$}] (G) at (3398, -1915);
\endtikzpicture%
\hfill
\begin{tikzpicture}[x=1.0cm,y=1.0cm,scale=0.3]
\clip(-8.3,-2.52) rectangle (12.12,9.5);
\draw[very thick](1.62,1.32) circle (0.8655634003352958cm);
\draw[very thick](3.2328243459690333,2.6884570208222103) circle 
(1.2495907285139565cm);
\draw[dashed](1.9,2.32) circle (0.5813776741499452cm);
\draw[very thick](1.0440966422730462,3.051087438658409) circle 
(0.9588072850579203cm);
\draw[dashed](2.121868047303161,3.590154374218259) circle 
(0.7080088329765607cm);
\draw[very thick](2.214411257544721,5.303385656555337) circle 
(1.5793975059430259cm);
\draw[dashed](0.6232554848541616,4.3250721755571675) circle 
(0.9611058386547104cm);
\draw[very thick](-3.450131067664543,5.241404237126853) circle 
(4.040749733961258cm);
\draw[dashed](-0.17927992929535486,2.771278134068165) circle 
(0.7877049042019487cm);
\draw[very thick](-0.14071528239912867,1.5482192694258121) circle 
(0.9182671086615662cm);
\draw[dashed](0.8621000298939481,2.0109564826755766) circle 
(0.5018880317616949cm);
\draw[dashed](3.6928600787966377,4.217869196206713) circle 
(0.9837482006308076cm);
\draw[very thick](6.5218757202127815,5.012002007088517) circle 
(2.777408184525129cm);
\draw[dashed](4.8864789619475495,2.4965209652940263) circle 
(1.1109512456575303cm);
\draw[very thick](5.537606654114753,0.08517185708964223) circle 
(2.227349755827358cm);
\draw[dashed](2.577704142011834,1.5352899408284022) circle 
(0.45546984649380334cm);
\draw[very thick](2.944049976286953,0.9658318427967093) circle 
(0.5050362579465297cm);
\draw[dashed](3.6123015282077824,1.3652049802178494) circle 
(0.593072058477195cm);
\end{tikzpicture}
\caption{{\it Left:} The exterior intersection angle $\alpha$
of two intersecting circles and the associated
  kite built from centers and intersection points; {\it right:} 
An example of a circle packing (black circles) with additional 
circles (dashed) added to obtain an orthogonal circle pattern}\label{figKite}
\end{figure}
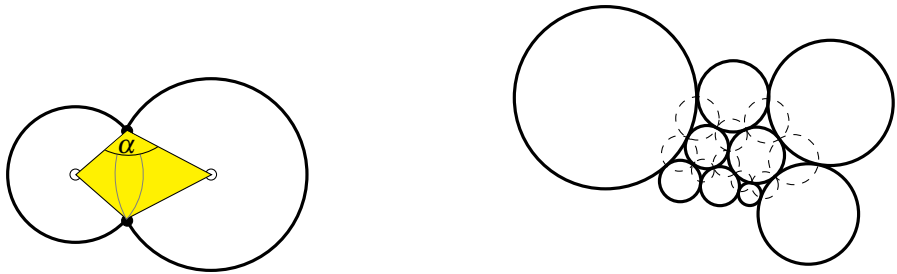

Circle patterns for given combinatorics and intersection angles have been 
introduced and studied for example in~\cite{Sch92,BS93,Ri94,Sch97,BS02}.
We recall some results in Section~\ref{CircelPatterns}.
Moreover, circle patterns generalize planar {\em circle packings} which are 
configurations of discs corresponding to a triangulation of a planar domain 
where vertices correspond to discs and for each edge the two corresponding 
discs touch. Every such circle packing can be understood as an orthogonal 
circle pattern by 
adding circles corresponding to the triangular faces through the three touching 
points, see Figure~\ref{figKite}~(right) for an example. The corresponding 
intersection angles are all $\pi/2$.
More details can for example be found in the monograph~\cite{St05} and in the
references therein. 

Our first question in this article concerns a discrete analogue of Liouville's 
theorem in complex analysis, that is ``Any bounded entire function is 
constant''. To this end, we investigate infinite circle 
patterns whose corresponding kite patterns fill the whole complex plane $\C$ or 
the unit disc $\D$. Obviously, the underlying b-quad-graph is also infinite. We 
always assume that these infinite circle patterns are {\em locally finite} 
in $\C$ or $\D$, that is, for every compact subset of $\C$ or $\D$, 
respectively, there are only finitely many kites of the corresponding kite 
pattern which intersect this set.

For such infinite circle patterns, rigidity of these configurations is a 
natural question. In particular, given two infinite locally finite embedded 
circle patterns ${\mathscr C}$ and $\widetilde{\mathscr C}$ with the same 
combinatorics and the same intersection angles, they should be related by an 
Euclidean similarity in $\C$ or a hyperbolic isometry in $\D$, respectively.
In~\cite{He99}, He proved rigidity for 
infinite embedded locally finite Thurston-type circle patterns with exterior 
intersection angles in $[\pi/2,\pi]$. These are patterns where circles 
correspond to vertices of a 
triangulation and for each edge the two corresponding circles intersect with 
given exterior intersection angles (with an additional condition
for the case $\alpha(f)=\pi/2$). Note that Thurston-type circle 
patterns are in general not circle patterns with respect to our definition used 
in this paper, but some special classes of circle patterns are included. In 
particular, He's result holds for orthogonal circle patterns, that is circle 
patterns as defined above with $\alpha\equiv\pi/2$.
Adapting ideas of He's proof,
we show that rigidity holds for locally finite circle patterns in $\D$ without 
restrictions on intersection angles.

\begin{theorem}[Rigidity of infinite circle patterns in the hyperbolic 
disc]\label{theoRidhyp}
 Let $\mathscr D$ be a b-quad-graph 
and let $\alpha:F({\mathscr D})\to (0,\pi)$
be an admissible labelling. Let ${\mathscr C}$ and $\widetilde{\mathscr C}$ be 
two locally finite embedded infinite circle patterns for $\mathscr D$ and 
$\alpha$ which are contained in $\D$. Assume that the kites of the 
corresponding kite patterns ${\mathscr K}$ and $\widetilde{\mathscr K}$ cover 
the whole unit disc $\D$ in each case.
Then there is a hyperbolic isometry $f:\D\to\D$ such that $\widetilde{\mathscr
C}=f({\mathscr C})$.
\end{theorem}

The proof relies on a generalized maximum principle for the radius function of 
the circle pattern in the hyperbolic disc which is shown in 
Section~\ref{SecRigidity}.

For infinite circle patterns covering the whole complex plane we adapt further 
ideas of He~\cite{He99} and prove rigidity for the class of infinite circle 
patterns whose intersection angles are uniformly bounded from below.

\begin{theorem}[Rigidity of infinite Euclidean circle 
patterns]\label{theoRidEuc}
 Let $\mathscr D$ be a b-quad-graph, 
$\alpha_0\in(0,\pi/2]$ and let $\alpha:F({\mathscr D})\to (\alpha_0,\pi)$
be an admissible labelling. Let ${\mathscr C}$ and $\widetilde{\mathscr C}$ be 
two locally finite embedded infinite circle patterns for $\mathscr D$ and 
$\alpha$.
Assume that the kites of the 
corresponding kite patterns ${\mathscr K}$ and $\widetilde{\mathscr K}$ cover 
the whole plane $\C$ in each case.
Then there is a Euclidean similarity $f:\C\to\C$ such that $\widetilde{\mathscr
C}=f({\mathscr C})$.
\end{theorem}

The proof is given in Section~\ref{secRigEuc}. Some parts of the proof rely 
heavily on estimates using vertex extremal length for graphs. This notion and 
some properties are introduced in Section~\ref{secVEL}.

\par\medskip

Recall that smooth conformal maps preserve intersection angles. Therefore,
pairs of circle patterns with the same combinatorics may be considered as 
discrete analogues of smooth conformal maps if all corresponding interesection 
angles agree.
More precisely, we consider a piecewise linear map between two kite 
patterns corresponding to (finite) embedded circle 
patterns as a {discrete conformal map}.

\begin{definition}\label{deffdiamond}
Let ${\mathscr C}$ and $\widetilde{\mathscr C}$ be two  embedded circle 
patterns with the same combinatorics and the same intersection angles. Let 
${\mathscr K}$ and $\widetilde{\mathscr K}$ be the corresponding kite patterns. 
Split every kite of these patterns into two triangles along the symmetry axis 
which connects its two white vertices. Define a global homeomorphism 
$f^\Diamond:{\mathscr K}\to \widetilde{\mathscr K}$ by the condition that the 
restriction of $f^\Diamond$ to each such triangle of ${\mathscr K}$ is an 
affine-linear map onto the corresponding triangle of $\widetilde{\mathscr K}$. 
In particular, $f^\Diamond$ maps black and white vertices of ${\mathscr K}$ to 
corresponding points of $\widetilde{\mathscr K}$. By abuse of notation we also 
denote this {\em discrete conformal map} as $f^\Diamond:{\mathscr C}\to 
\widetilde{\mathscr C}$.
\end{definition}

The second aim of this article is to study some sequences of discrete 
conformal maps and show that they approximate smooth conformal maps.

The convergence of discrete conformal maps based on circle packings was first 
conjectured by Thurston~\cite{Thu85} and then 
proven for regular hexagonal combinatorics, see~\cite{RS87,HeSch98}. 
Furthermore, He and Schramm studied in~\cite{HeSch96} the approximation of a 
conformal homeomorphism by circle packings with arbitrary combinatorics.

Circle patterns with the combinatorics of the 
square grid and their relations to smooth conformal maps have been 
studied in~\cite{Sch97,Ma05,LD07}. These regular circle patterns form a special 
case of {\em isoradial} circle patterns where all circles have the same radius.
For isoradial circle patterns with uniformly bounded intersection angles 
convergence has been shown in~\cite{Bue08}.
Using ideas from~\cite{HeSch96} and~\cite{RS87}, we prove in this article
convergence to a conformal map for the class of {\em q-bounded convex circle 
patterns}. These circle patterns satisfy the condition that all kites are 
convex and the 
quotients of lengths of the diagonals of the kites are uniformly bounded in 
$[1/q, q]$, see Definition~\ref{defqbound}. In particular, we obtain the 
following result.

\begin{theorem}\label{theoConv}
 Let $D$ and $\widetilde{D}$ be two simply connected bounded domains in $\C$. 
Let $p_0\in D$ be some `reference' point.
Let $({\mathscr D}_n)_{n\in\N}$ be a sequence of finite b-quad-graphs which are 
cell decompositions of a topological closed disc.

Let $q>1$. For every $n\in\N$ 
assume that ${\mathscr C}_n$ and $\widetilde{\mathscr C}_n$ are two embedded 
convex q-bounded
 circle patterns for ${\mathscr D}_n$ and some admissible labelling 
$\alpha_n$ whose kites all lie in $D$ and $\widetilde{D}$ respectively.

Let $(\delta_n)_{n\in\N}$ be a sequence of positive numbers such that 
$\delta_n\searrow 
0$ for $n\to\infty$. For each $n\in\N$ denote by $r_n(v)$ the radius of the 
circle in ${\mathscr C}_n$ for the white vertex $v$ and assume that 
$r_n(v)<\delta_n/2$ for all 
white vertices. Suppose further that the Euclidean distance of the subset 
covered by the kites of ${\mathscr C}_n$ to the boundary $\partial D$ is 
smaller than $\delta_n$, 
i.e.\ $d({\mathscr S}_n,\partial D)<\delta_n$ where ${\mathscr S}_n$ is the 
union of all kites of ${\mathscr K}_n$, and also  
$d(\widetilde{\mathscr S}_n,\partial \widetilde{D})<\delta_n$ where 
$\widetilde{\mathscr S}_n$ is the union of all kites of $\widetilde{\mathscr 
K}_n$. 
Finally, let $p_0$ be covered by a kite for every $n\in\N$ and let the 
closure of the image points $\overline{(f_n^\Diamond(p_0))}_{n\in\N}$ be 
compact in $\widetilde{D}$.

Then a subsequence of $(f_n^\Diamond)_{n\in\N}$ converges uniformly on compact 
subsets of $D$ to a conformal homeomorphism $f:D\to\widetilde{D}$.
\end{theorem}

The proof is presented in Section~\ref{secConv} and also relies on the 
generalized maximum principle for the radius function in
Section~\ref{SecRigidity} and on topological properties studied in 
Section~\ref{secTopProp}.

Note that convergence issues on q-bounded kite patterns for 
Dirichlet problems with respect to the linear 
Laplacian (see~\eqref{eqdeflap} below) have recently been investigated 
in~\cite{WB}.

\section{Some properties of circle patterns}\label{CircelPatterns}

We start by introducing some useful notations.

Recall that the combinatorics for a circle pattern is given via a b-quad-graph 
$\mathscr D$, that is a strongly regular cell decomposition of (a part of) $\C$ 
whose 2-cells (faces) are  
embedded quadrilaterals (respecting the counterclockwise orientation of $\C$), 
so in particular there are exactly four edges incident to
each face. Moreover, its 1-skeleton is a bipartite graph with bicolored 
vertices.
It is called {\em simply connected} if it is the cell decomposition of a simply 
connected domain of $\C$.
The two sets of vertices (colored white and black respectively) give rise to 
two planar graphs $G$ and $G^*$ as follows.
The vertices $V(G)$ are all white vertices of 
$V(\mathscr D)$ and correspond to centers of circles. The edges $E(G)$
correspond to faces of $\mathscr D$, that is
two vertices of $G$ are connected by an edge if and only if they
are incident to the same face of $\mathscr D$, see 
Figure~\ref{figLeft}. 
\begin{figure}
\begin{center}
\includegraphics[width=0.45\textwidth]{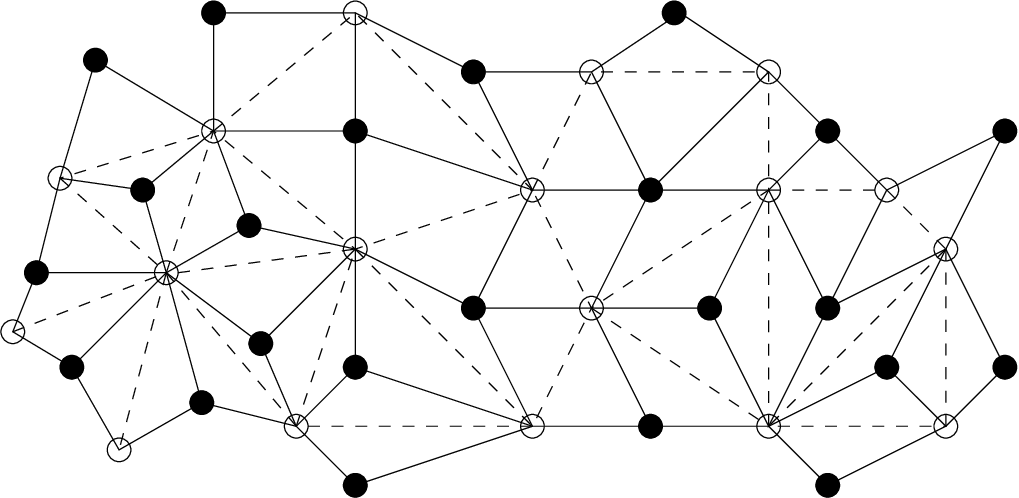}
\end{center}
\caption{An example of a b-quad-graph $\mathscr D$ (black edges and
bicolored vertices) and its associated graph $G$ (dashed edges and white 
vertices)}\label{figLeft}
\end{figure}

We denote by $V_{\text{int}}(G)$ the set of interior vertices of $G$ and by 
$V_\partial(G)$ the set of boundary vertices.
The dual graph $G^*$ is constructed analogously by taking for
$V(G^*)$ all black vertices of $V(\mathscr D)$ which correspond to 
intersection
points. Faces of $\mathscr D$ or edges of $G$ or $G^*$ correspond to
kites of intersecting circles. Furthermore, by abuse of
notation, every labelling $\alpha:F({\mathscr D})\to (0,\pi)$ of the faces of 
$\mathscr D$ can also be understood as a function defined on the edges $E(G)$ 
or $E(G^*)$. We denote the value of this function by $\alpha(f)$ 
or $\alpha_e$.

In this article we are concerned with embedded circle patterns. Nevertheless,
some useful properties naturally include the more general case of immersed
circle patterns. 

\begin{definition}\label{defcircpattern}
Let $\mathscr D$ be a b-quad-graph and let $\alpha:E(G)\to (0,\pi)$
be an admissible labelling.
  An {\em immersed circle pattern for  $\mathscr D$ (or $G$) and
    $\alpha$} is an indexed collection ${\mathscr 
    C}=\{{\mathscr C}(v):v\in V(G)\}$ of circles in $\C$ and an indexed
  collection ${\mathscr K}=\{K_e=K(e):e\in E(G)\}$
of closed kites, which all carry the same orientation, such that the following
conditions hold.
\begin{enumerate}[(1)]
  \item\label{intersectionPoint}
If $v_1,v_2\in V(G)$ are incident vertices in $G$,
the corresponding circles ${\mathscr C}(v_1),{\mathscr C}(v_2)$ intersect 
with exterior intersection angle $\alpha_{[v_1,v_2]}$.
 Furthermore, the kite $K_{[v_1,v_2]}$
    is bounded by the centers of the circles ${\mathscr C}(v_1),{\mathscr 
C}(v_2)$, the two intersection points, and the corresponding edges, as in
Figure~\ref{figKite}~(left).
 The intersection points
are associated to black vertices of $V(\mathscr D)$ or to vertices of $V(G^*)$.
\item The kite pattern is locally isomorphic to the b-quad-graph $\mathscr 
D$ at every interior vertex. 
  \end{enumerate}
\end{definition}

There are also other definitions for
 circle patterns, for example associated to a {\em Delaunay
 decomposition} of a domain in $\C$ for a given set of points (=vertices). This 
is a cell decomposition such
that the boundary of each face is a polygon with straight edges which is
inscribed in a circular disk, and these disks have no vertices in their
interior. The corresponding embedded circle pattern can be associated to the
graph $G^*$.
The Poincar{\'e}-dual decomposition of a Delaunay decomposition with the
centers of the circles as vertices and straight edges is a {\em
    Dirichlet decomposition} (or {\em Voronoi diagram}) and corresponds
  to the graph $G$.

Furthermore the definition of immersed circle patterns can be extended allowing
cone-like singularities in the vertices; see~\cite{BS02} and the references
therein.


Our study of a planar circle pattern ${\mathscr C}$ is based on
characterizations and properties of its
{\em radius function} $r_{\mathscr C}=r$ which assigns to every vertex $z\in V(G)$
the radius $r_{\mathscr C}(z)=r(z)$ of the corresponding circle ${\mathscr 
C}(z)$. The
index $\mathscr C$ will be dropped whenever there is no confusion likely.

The following proposition specifies a necessary and sufficient condition for a
radius function to originate from a planar circle pattern, see also~\cite{BS02}.
For the special case of orthogonal
circle patterns with the combinatorics of the square grid, there are also
other characterizations, see for example~\cite{Sch97}.

\begin{proposition}\label{propRadius}
Let $G$ be a graph constructed from a b-quad-graph $\mathscr
      D$ and let $\alpha$ be an admissible labelling. 

      Suppose that ${\mathscr C}$ is an immersed circle
      pattern for ${\mathscr D}$ and $\alpha$ with radius function
      $r=r_{\mathscr C}$. Then for 
      every interior vertex $v_0\in V_{int}(G)$ we have
      \begin{equation} \label{eqFgen}
        \Biggl(\sum_{[v,v_0]\in E(G)} 
        f_{\alpha_{[v,v_0]}}(\log r(v)-\log r(v_0))\Biggr) -\pi=0,
        \end{equation}
       where
        \begin{equation}\label{defftheta}
         f_\theta(x):=\frac{1}{2i}\log
        \frac{1-e^{x-i\theta}}{1-e^{x+i\theta}},
        \end{equation}
        and the branch of the logarithm is chosen such that
        $0<f_\theta(x)<\pi$. 

Conversely, suppose that $\mathscr D$ is simply connected and
      that $r:V(G)\to(0,\infty)$ satisfies~\eqref{eqFgen} for
      every $v\in  V_{int}(G)$. Then there is an immersed circle pattern
      for $G$ and $\alpha$
whose radius function coincides with $r$.
This pattern is unique up to isometries of $\C$.
\end{proposition}

Note that $2f_{\alpha_{[v,v_0]}}(\log r(v)-\log r(v_0))$ is the angle at $v_0$
of the kite
with edge lengths $r(v)$ and $r(v_0)$ and angle $\alpha_{[v,v_0]}$, as in
Figure~\ref{figKite}~(left). Equation~\eqref{eqFgen} is the closing
condition for the chain of kites corresponding to the edges incident
to $v_0$.

\begin{proof}
The statements are shown in a more general context in~\cite[Section~3]{BS02} 
in the case of finite circle patterns. As~\eqref{eqFgen} is only a local 
condition, it still holds for infinite patterns. 

Also, given any radius 
function for the simply connected b-quad-graph $\mathscr D$ and the admissible 
labelling $\alpha$, we can subsequently construct a corresponding circle 
pattern. For each quad with white vertices $v_1$ and $v_2$ build a kite as in 
Figure~\ref{figKite}~(left) with sides of lengths 
$r(v_1)$ and $r(v_2)$ and angle 
$\alpha_{[v_1,v_2]}$ opposite to the symmetry axis of the kite. Because
corresponding sides have the same length, these kites may be glued together to
form a cellular decomposition equivalent to $\mathscr D$. 
Condition~\eqref{eqFgen} and the admissible labelling assure that the 
resulting kite pattern can locally be embedded in the plane. Also, we obtain by 
construction a corresponding circle pattern if we add circles of radius $r(v)$ 
through all vertices of the kites which correspond to black vertices incident to 
$v$ in $\mathscr D$. 
In particular, we can start with any kite corresponding to a face $f$ of 
$\mathscr D$ and embed it into the plane. This fixes the rotational and 
translational freedom. Then we glue kites for faces adjacent to $f$ and 
successively proceed. Due to our assumptions on the labelling and to 
condition~\eqref{eqFgen}, the (finite) chain of kites corresponding to the 
edges incident to any fixed vertex will close up. Proceeding in this way for 
all faces of the simply connected b-quad-graph $\mathscr D$ finally leads 
to an immersed (maybe infinite) circle pattern. 
  
\end{proof}

For further use we mention some properties of $f_\theta$, see for
example~\cite{Spr03}.

\begin{lemma}\label{lemPropf}
  \begin{enumerate}[(1)]
    \item The derivative of $f_\theta$ is
    $ f_\theta'(x)=\frac{\sin\theta}{2(\cosh x -\cos \theta)}>0$.
    So $f_\theta$ is strictly increasing.
    \item The function $f_\theta$ satisfies the functional equation
      $f_\theta(x)+f_\theta(-x)=\pi-\theta$.
    \item For $0<y<\pi-\theta$ the inverse function of $f_\theta$ is
      $f_\theta^{-1}(y)=\log\frac{\sin y}{\sin(y+\theta)}$.
\item The integral of $f_\theta$ is 
\begin{equation*}
 F_\theta(x) = \int_{-\infty}^x f_\theta(\xi)d\xi = \Im \Li(e^{x+i\theta}),
\end{equation*}
where $\Li(z)= -\int_0^z\frac{\log(1-\zeta)}{\zeta}d\zeta$ is the 
{\em dilogarithm function}, see for example \cite{L81}.
\end{enumerate}
\end{lemma}

In analogy to smooth harmonic functions, the radius function of a planar
circle pattern satisfies a Dirichlet principle and a maximum principle.

\begin{theorem}[Dirichlet Principle]\label{theoDirichlet}
Let $G$ be a graph constructed from a finite simply connected 
b-quad-graph $\mathscr D$ and let $\alpha$ be an admissible labelling.
  
  Let $r:V_\partial(G)\to(0,\infty)$ be some positive function on the
  boundary vertices of $G$. Then $r$ can be extended to $V(G)$ in such
  a way that equation~\eqref{eqFgen} holds at every interior vertex $z\in
  V_{int}(G)$ if and only if
there exists any immersed circle pattern for $G$ and $\alpha$.
  If it exists, the extension is unique.
\end{theorem}
See for example~\cite[Theorem~2.6]{Bue08} for a proof.

\begin{lemma}[Maximum Principle]\label{lemMaxEuc}
Let $G$ be a graph constructed from a b-quad-graph $\mathscr D$ and let 
$\alpha$ 
be an admissible labelling.
      Suppose ${\mathscr C}$ and $\widetilde{\mathscr C}$ are two immersed 
circle
      patterns for $G$ and $\alpha$ with radius functions
      $r$ and $\tilde r$ resp.
      Then the maximum (or minimum) of the quotient $r/ \tilde r$ is
attained at a boundary vertex.
\end{lemma}
A proof can be found in~\cite[Lemma~2.1]{He99}.
If there exists an
isoradial planar circle pattern for $G$ and $\alpha$, the usual maximum
principle for the radius function follows by taking $\tilde r\equiv 1$.

\section{Maximum principles and rigidity of circle patterns 
in the hyperbolic plane}\label{SecRigidity}

For a circle pattern $\mathscr C$ in the unit disc $\D$ for the graph $G$ 
denote 
by $r_{\text{hyp}}:V(G)\to\R$ the function which assigns to each vertex the
{\em hyperbolic radius} of the corresponding circle.

\begin{lemma}[Maximum Principle in the hyperbolic plane]\label{lemMaxHyp}
Let $G$ be a graph constructed from a b-quad-graph $\mathscr D$ and let 
$\alpha$ 
be an admissible labelling.
Let ${\mathscr C}$ and ${\mathscr C}^*$ be two circle patterns for
$G$ and $\alpha$ which are contained in $\D$. 

 If the inequality $r_{\text{hyp}}^*(v)\geq
r_{\text{hyp}}(v)$ is satisfied for each boundary vertex, then it holds for 
all vertices of $G$.
\end{lemma}

Our proof is based on a variational principle for circle patterns by Springborn 
and Bobenko, see~\cite{BS02,Spr03}. We will briefly present some useful 
facts. 

First we introduce some notation.
For a hyperbolic circle pattern for $G$ denote by $c_{\text{hyp}}:V(G)\to\C$ 
the function assigning the hyperbolic center of circle to the vertex of the 
corresponding circle. Also, instead of the hyperbolic radius function 
$r_{\text{hyp}}$ we will often consider the function  
\begin{equation}
 \rho:V(G)\to\R, \qquad \rho(v)=\log\tanh\frac{r_{\text{hyp}}(v)}{2}.
\end{equation}
Note that positive hyperbolic radii $r_{\text{hyp}}(v)$ correspond to 
negative $\rho(v)$.

Consider two intersecting circles in $\D$ corresponding to the edge 
$e=[v_0,v_1]$. Let $\alpha_e= \theta \in(0,\pi)$ be
the exterior intersection angle of the two circles ${\mathscr C}(v_0)$ and 
${\mathscr C}(v_1)$. 
Consider the hyperbolic triangle formed by the two centers 
$c_{\text{hyp}}(v_0)$ and $c_{\text{hyp}}(v_1)$ and one of the 
intersection points of the two circles, see 
Figure~\ref{figsPhi}\subref{figPhi} for a schematic sketch. Then 
the angle at $c_{\text{hyp}}(v_0)$ in this triangle is 
\begin{equation}
 \varphi_{e}(\rho(v_0),\rho(v_1),\theta)= 
f_\theta(\rho(v_1)-\rho(v_0))- 
f_\theta(\rho(v_1)+\rho(v_0)),
\end{equation}
where the function $f_\theta:\R\to\R$ is defined in~\eqref{defftheta}.
\begin{figure}[t]
\subfloat[Definition of the angle 
$\varphi_{e}= \varphi_{e}(\rho(v_0),\rho(v_1),\theta)$ in a
hyperbolic triangle (schematic drawing, all lines are geodesics)]{\label{figPhi}
\hspace{-0.2cm}
\raisebox{1.2cm}{ \input{DefPhi.pspdftex}}
\hspace{0.7cm}
}
\hspace{0.1cm}
\subfloat[Generalization of $\varphi_{e}$ by 
$\varphi_{e}^g$ (schematic drawing, bold lines are geodesics)]{\label{figDefPhi}
\input{GenPhi.pspdftex}
}
\vspace{-1.5ex}
\caption{Definition of the angle functions for (generalized) hyperbolic circle 
patterns.}\label{figsPhi}
\end{figure}
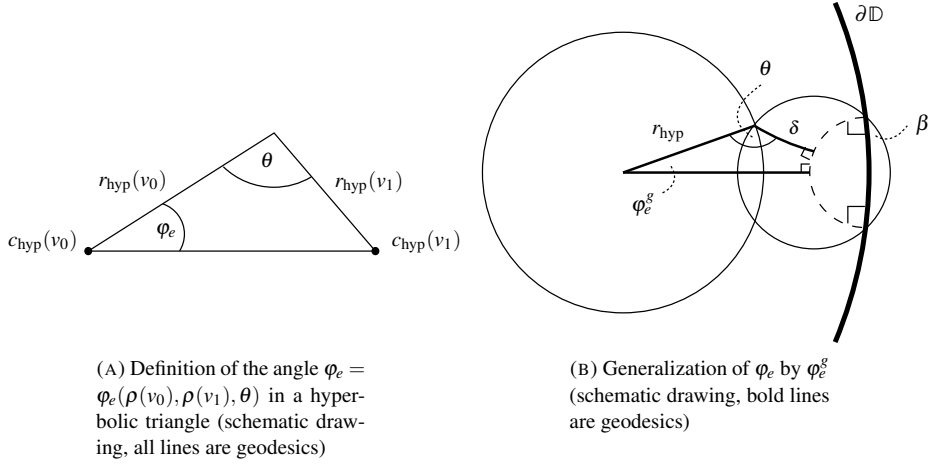

In a planar (hyperbolic) circle pattern these angles $\varphi_{e}$ add up 
to $\pi$ around an interior vertex $v_0$. Therefore we have analogously as in 
Proposition~\ref{propRadius} the following characterization of hyperbolic 
circle patterns.

\begin{proposition}[\cite{BS02,Spr03}]\label{propShyp}
 Let $G$ be a graph constructed from a b-quad-graph $\mathscr
      D$ and let $\alpha$ be an admissible labelling. 
\begin{enumerate}[(i)]
 \item Suppose that ${\mathscr C}$ is a planar hyperbolic circle
      pattern for ${\mathscr D}$ and $\alpha$ with hyperbolic radius function
      $r=r_{\text{hyp}}$ and $\rho= \log\tanh\frac{r_{\text{hyp}}}{2}$. Then 
for      every interior vertex $v_0\in V_{int}(G)$ we have
\begin{equation}\label{eqhyprho}
 2\pi -2\sum_{[v_0,v]\in E(G)} (f_{\alpha_{[v_0,v]}}(\rho(v)-\rho(v_0))- 
f_{\alpha_{[v_0,v]}}(\rho(v)+\rho(v_0))) =0.
\end{equation}
\item Conversely, suppose that $\mathscr D$ is simply connected and assume
      that the variables $\rho= \log\tanh\frac{r}{2}\in (-\infty, 0)$ for a 
hyperbolic radius 
function $r:V(G)\to(0,\infty)$ satisfy~\eqref{eqFgen} for every $v_0\in  
V_{int}(G)$. Then there is a planar hyperbolic circle pattern for $G$ and 
$\alpha$ whose hyperbolic radius function coincides with $r$.
This pattern is unique up to isometries of $\D$.
\item For given boundary values $\rho: V_\partial(G)\to (-\infty, 0)$ there 
exists 
a unique solution $\rho:V_{\text{int}}(G)\to\R$ of equations~\eqref{eqhyprho} 
which is 
the minimizer of the strictly convex functional $S_{\text{hyp}}: 
\R^{N}\to\R$, where $N=|V_{int}(G)|$ and
\begin{multline} \label{eqDefShyp}
 S_{\text{hyp}}(\rho) = \sum_{e=[v_l,v_r]} \left(
\Im \Li(e^{\rho(v_l)-\rho(v_r)+i\alpha_e}) + 
\Im\Li(e^{\rho(v_r)-\rho(v_l)+i\alpha_e}) \right. \\
\shoveright{\left.+\Im\Li(e^{\rho(v_l)+\rho(v_r)+i\alpha_e}) +
\Im\Li(e^{-\rho(v_l)-\rho(v_r)+i\alpha_e})\right)}\\
+ \sum_{v\in V_{int}(G)} 2\pi\rho(v).
\end{multline}
The first sum is taken over all edges $e\in E(G)$.
\end{enumerate}
\end{proposition}

Note in particular that for all $v_0\in V_{int}(G)$ and $\rho_0= 
\rho(v_0)$ we have
\[\frac{\partial S_{\text{hyp}}(\rho)}{\partial \rho_0} = 2\pi 
-2\sum_{[v_0,v]\in E(G)} (f_{\alpha_{[v_0,v]}}(\rho(v)-\rho(v_0))- 
f_{\alpha_{[v_0,v]}}(\rho(v)+\rho(v_0))).\]

\begin{remark}
 The statements of the preceding proposition can be generalized to hyperbolic 
circle patterns with cone singularities at the vertices, see~\cite{BS02,Spr03}.
\end{remark}

Now we are ready to prove the Maximum Principle in the hyperbolic plane.

\begin{proof}[of Lemma~\ref{lemMaxHyp}.]
 First note that $\rho(r)=\log\tanh\frac{r}{2}$ is strictly 
monotonically increasing in (the hyperbolic radius) $r$. 

By Proposition~\ref{propShyp}, circle patterns in $\D$ correspond uniquely to 
critical points of the functional $S_{\text{hyp}}$. Furthermore, 
$S_{\text{hyp}}$ is strictly convex and its unique critical point is a minimum.

Let $v_1,\dots,v_n$ be a numbering of the interior vertices of $G$ 
where $n=|V_{int}(G)|$. Let $\rho, \rho^*:V(G)\to 
(-\infty,0)$ correspond to the radii as above.
Consider in $\R^n$ the $n$-dimensional closed interval $U=\{u\in\R^n\ |\ 
u_i\leq \rho^*(v_i)\}$. 
We will show that for fixed boundary values $\rho|_{\partial V}$ the 
gradient of $S_{\text{hyp}}(\cdot, \rho|_{\partial V}):\R^n\to\R$ is pointing 
into the complement of $U$ on the boundary $\partial U$ or is contained in 
$\partial U$ (or more precisely in one of the hyperplanes $H_i=\{u\in\R^n\ |\ 
u_i= \rho^*(v_i)\}$).

For $u_i=\rho^*(v_i)$ the $i$th component of the gradient of $S_{\text{hyp}}$ 
is 
\begin{multline*}
 \frac{\partial S_{\text{hyp}}(\rho)}{\partial u_i}(u_1,\dots, 
u_i=\rho^*(v_i),\dots u_n) \\
= 2\pi -2\sum_{[v_i,v_k]\in E(G)} (f_{\alpha_{[v_i,v_k]}}(u_k-\rho^*(v_i))- 
f_{\alpha_{[v_i,v_k]}}(u_k+\rho(v_i))) \\
\geq 2\pi 
-2\sum_{[v_i,v_k]\in E(G)} (f_{\alpha_{[v_i,v_k]}}(\rho^*(v_k)-\rho^*(v_i))- 
f_{\alpha_{[v_i,v_k]}}(\rho^*(v_k)+\rho(v_i))) \\
\qquad =\frac{\partial S_{\text{hyp}}(\rho)}{\partial u_i} (\rho^*(v_1),\dots, 
\rho^*(v_n))=0.
\end{multline*}
The sum is taken over all edges $[v_i,v_k]\in E(G)$ incident to $v_i$.
The inequality follows from the fact that 
\begin{align*}
\frac{\partial}{\partial u_k} \varphi_{e}(u_i,u_k,\theta) &= 
\frac{\partial}{\partial u_k}(f_\theta(u_k-u_i)- f_\theta(u_k+u_i))\\
&=\frac{\sin\theta}{2(\cosh(u_k-u_i)-\cos\theta)} - 
\frac{\sin\theta}{2(\cosh(u_k+u_i)-\cos\theta)} >0.
\end{align*}
Furthermore the inequality is strict if at least for one edge $[v_i,v_k]\in 
E(G)$ we have a strict inequality $u_k<\rho^*(v_k)$.
Therefore the minimum of $S_{\text{hyp}}$ with boundary conditions 
$\rho|_{\partial V}$ in contained in $U$.
  \end{proof}

Similarly as He in~\cite{He99} we now generalize this maximum principle to 
circle patterns which only intersect the unit disc $\D$. For this reason, we 
generalize the hyperbolic functional $S_{\text{hyp}}$ for circles intersecting 
the boundary $\partial \D$.

\begin{remark}
 In the Poincar\'e disc model of the hyperbolic geometry the geodesics are the 
intersections of circles orthogonal to $\partial \D$ with $\D$.
Consider a circle which intersects $\partial \D$ in two points with 
exterior intersection angle $\beta$. Then all points on the corresponding 
circular arc in $\D$ have the same 
distance to one fixed geodesic which intersects $\partial \D$ in the same two 
points, see Figure~\ref{figCircLine}. This distance $\delta$ satisfies
\begin{align}
 i\beta = \log \tanh\left(\frac{\delta}{2}+i\frac{\pi}{4}\right).
\end{align}
These curves are also called {\em hypercycles} and may be 
interpreted as generalization of circles where all points have the same 
distance to one fixed point (which is the center).
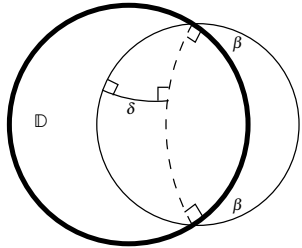
\begin{figure}[htb]
\centering{\input{IntersectionCircle.pspdftex}}
 \caption{A circular arc corresponding to a hypercycle in the Poincar\'e disc 
model with exterior intersection angle $\beta$ and distance $\delta$ to the 
corresponding geodesic (dashed line).}\label{figCircLine}
\end{figure}
\end{remark}

First, we generalize the angle function $\varphi_{e}$. Consider
a circular arc intersecting $\partial \D$ in some exterior angle 
$\beta\in[0,\pi)$ which intersects a circle contained in $\D$ in an angle 
$\theta\in(0,\pi)$. Let $c=c_{\text{hyp}}$ be the center of this circle and 
$r=r_{\text{hyp}}$ its radius. Set $\rho=\log\tanh\frac{r}{2}$ as above. Then 
the angle $\varphi_{e}^g(\rho,\beta,\theta)$ at $c$ is given by (see 
Figure~\ref{figsPhi}\subref{figDefPhi}):
\begin{align}
 \varphi_{e}^g(\rho,\beta,\theta) &= 
\frac{1}{2i}\log\left(\frac{\cos\beta 
-\cosh(\rho+i\theta)}{\cos\beta -\cosh(\rho-i\theta)} \right) \\
& =f_\theta(i\beta-\rho)- f_\theta(i\beta+\rho)\label{eqDefPhig}
\end{align}
In the first equation, the logarithm is chosen such that $\varphi_{e}^g\in 
(0,\pi)$. For the second equation we use a suitable generalization (in fact 
extention) of the function $f_\theta$ 
 which maps 
$D_\theta= \R\times(-i\pi,i\pi)\setminus B_\theta$ with $B_\theta= 
\{0\}\times(-i\pi,-i\theta]\cup \{0\}\times [i\theta,i\pi)$ to a simply 
connected 
domain in $\{x+iy\ :\ x\in(-\frac{\theta}{2}, \pi-\frac{\theta}{2}),\ y\in \R\} 
\subset\C$.

Note that there is an analogous formula for the distance $\delta$, namely
\begin{align}
-i\delta&= 
\frac{1}{2i}\log\left(\frac{\cosh\rho
-\cos(\beta+\theta)}{\cosh\rho -\cos(\beta-\theta)} \right) \\
& =f_\theta(\rho-i\beta)- f_\theta(\rho+i\beta)
\end{align}

\begin{remark}
Note that $f_\theta$ is one-to-one on $D_\theta$.

Furthermore we have for $\beta\in[0,\pi)$:
\begin{align}
& \lim_{x\to -\infty} f_\theta(x\pm i\beta) =0 ,\\
& \lim_{x\to +\infty} f_\theta(x\pm i\beta) =\pi-\theta, \\
& f_\theta(x+i\beta) + f_\theta(-x-i\beta) = f_\theta(-x+i\beta)+ 
f_\theta(x-i\beta) = \pi-\theta. \label{eqSumf}
\end{align}
\end{remark}

\begin{definition}
 Let $G$ be a graph constructed from a b-quad-graph $\mathscr
      D$ and let $\alpha$ be an admissible labelling. Suppose that ${\mathscr 
C}$ is a planar {\em Euclidean} circle pattern for ${\mathscr D}$ and 
$\alpha$ such that all circles intersect the hyperbolic disc $\D$. Then this 
circle pattern is called {\em generalized hyperbolic circle pattern} for 
${\mathscr D}$ and $\alpha$. To this circle pattern we associate a functional 
$\rho:V(G)\to\C$ which is $\rho(v)= \log\tanh\frac{r_{\text{hyp}}(v)}{2}$ for 
circles contained in $\D$ with hyperbolic radius $r_{\text{hyp}}(v)$. For 
circles intersecting $\partial \D$ with exterior intersection angle 
$\beta\in[0,\pi)$ we set $\rho(v)=i\beta$.
\end{definition}

We immediately deduce from the definition of $\varphi_{e}^g$ 
in~\eqref{eqDefPhig} the 
following generalization of Proposition~\ref{propShyp}~(i).

\begin{proposition}\label{propGenCirc}
 Let $G$ be a graph constructed from a b-quad-graph $\mathscr
      D$ and let $\alpha$ be an admissible labelling. 
Suppose that ${\mathscr C}$ is a generalized hyperbolic circle
      pattern for ${\mathscr D}$ and $\alpha$ with associated function $\rho$. 
Then for every interior vertex $v_0\in V_{int}(G)$ we have
\begin{equation}\label{eqhyprho2}
 2\pi -2\sum_{[v_0,v]\in E(G)} (f_{\alpha_{[v_0,v]}}(\rho(v)-\rho(v_0))- 
f_{\alpha_{[v_0,v]}}(\rho(v)+\rho(v_0))) =0.
\end{equation}
\end{proposition}

Now we generalize the hyperbolic functional $S_{\text{hyp}}$. We only change 
the terms for all edges connecting an interior vertex of $G$ to a boundary 
vertex with intersection angle $\beta\in[0,\pi)$. In other words, we change the 
term for edges where the function $\rho$ is real and negative on one vertex 
and purely imaginary or zero on the other vertex.

For $x\leq 0$ define
\begin{align*}
 F_{\beta,\theta} (x) &:=2\int_{-\infty}^x (f_\theta (\eta+i\beta) 
+f_\theta(\eta-i\beta))d\eta -2(\pi-\theta)x.
\end{align*}
For $x>0$ we set $F_{\beta,\theta} (x) = F_{\beta,\theta} (-x)$.

There is a relation to the $\Li$-function (at least for $\beta\not=\theta$) as 
in the finite case:
\begin{equation*}
 \int_{-\infty}^x f_\theta (\eta+i\beta)d\eta 
= \frac{1}{2i} (\Li(\text{e}^{x+i\beta+i\theta})- 
\Li(\text{e}^{x+i\beta-i\theta})).
\end{equation*}
By equation~\eqref{eqSumf} we can write
\begin{align*}
 -2(\pi-\theta)x &= \int_x^0 (f_\theta(\eta+i\beta) +f_\theta(-\eta-i\beta)
+f_\theta(\eta-i\beta) +f_\theta(-\eta+i\beta))d\eta \\
&= \int_x^{-x} 
f_\theta(\eta+i\beta)d\eta + \int_x^{-x} f_\theta(\eta-i\beta)d\eta .
\end{align*}
Thus we finally obtain
\begin{align*}
F_{\beta,\theta} (x) &= \frac{1}{2i} \left(\Li(\text{e}^{x+i\beta+i\theta})- 
\Li(\text{e}^{x+i\beta-i\theta}) - \Li(\text{e}^{-x+i\beta-i\theta}) + 
\Li(\text{e}^{-x+i\beta+i\theta}) \right.\\
&\left.\quad + \Li(\text{e}^{x-i\beta+i\theta}) - 
\Li(\text{e}^{x-i\beta-i\theta}) - \Li(\text{e}^{-x-i\beta-i\theta}) + 
\Li(\text{e}^{-x-i\beta+i\theta}) \right).
\end{align*}
Moreover, we obtain again with~\eqref{eqSumf}
\begin{align*}
 \frac{\partial F_{\beta,\theta}}{\partial x} 
&= 2(f_\theta (x+i\beta) +f_\theta(x-i\beta) -(\pi-\theta))
= 2 (f_\theta(x+i\beta) -f_\theta (-x+i\beta)) \\
&= -2\varphi^g(x,\beta,\theta).
\end{align*}

Now we are able to define the generalized hyperbolic functional 
$S_{\text{hyp}}^g(\rho)$. First, let $V_{\partial,\beta}(G)$ be the set of 
boundary vertices which correspond to circles intersecting $\partial \D$. Let 
$E_\beta(G)$ be the set of all edges which are incident to one interior vertex 
and one vertex of $V_{\partial,\beta}(G)$. Let $n=|V_{int}(G)|$, 
$k_\beta=|V_{\partial,\beta}(G)|$ and $k_0=|V_\partial(G)|-k_\beta$. 
We define $S_{\text{hyp}}^g:\R^{n+k_0}\times (i\R)^{k_\beta}\to\R$ similarly 
as in equation~\eqref{eqDefShyp} as a sum over all edges. The summands are the 
same as above in~\eqref{eqDefShyp} for all edges which are not incident to a 
vertex in $V_{\partial,\beta}(G)$. Otherwise we take the summand 
$F_{\beta(v),\alpha_{[v,v_0]}}$ for 
edges $[v,v_0]$ with one interior vertex $v_0$ and the boundary vertex $v$ with 
exterior intersection angle $\beta(v)$.

Note the following properties of $F_{\beta,\theta}$.
\begin{enumerate}[(i)]
 \item 
 \vspace{-1ex}
\begin{align*}
 \frac{\partial^2 F_{\beta,\theta}}{\partial x^2} 
&= \frac{\sin\theta}{\cosh(x+i\beta)-\cos\theta} + 
\frac{\sin\theta}{\cosh(-x+i\beta)-\cos\theta} \\
&= \frac{2\sin\theta(\cosh x \cos\beta -\cos\theta)}{(\cosh x \cos\beta 
-\cos\theta)^2 + (\sinh x\sin\beta)^2}.
\end{align*}
Thus for fixed boundary values $\rho|_{\partial V}$ (in particular for fixed 
$\beta(v)$) the functional 
$S_{\text{hyp}}^g$ is not necessarily convex.
\item If for all edges $e=[v,v_0]\in E_\beta$ with $v\in V_{\partial,\beta}$ we 
have $\cos\alpha_e< \cos\beta(v)$ 
then $S_{\text{hyp}}^g$ is strictly convex and any critical point must be 
a minimum. In this case existence can be proven along the same lines as for the 
original function $S_{\text{hyp}}$ in~\cite{BS02,Spr03}.
\end{enumerate}

As an immediate consequence Proposition~\ref{propGenCirc} can now be 
formulated as follows.

\begin{corollary}
 Let $G$ be a graph constructed from a b-quad-graph $\mathscr
      D$ and let $\alpha$ be an admissible labelling. 
Suppose that ${\mathscr C}$ is a generalized hyperbolic circle
      pattern for ${\mathscr D}$ and $\alpha$ with associated function $\rho^*$.

Then $\rho^*|_{V_{int}}$ is a critical point of the functional 
$S_{\text{hyp}}^g(\cdot,\rho^*|_{V_\partial})$.

Furthermore, if for all interior edges $e$ incident to a vertex 
$v\in V_{\partial,\beta}$ the inequality $\cos\alpha_{e}< \cos\beta(v)$ holds, 
then this critical point is the unique minimizer of 
$S_{\text{hyp}}^g(\cdot,\rho^*|_{V_\partial})$.
\end{corollary}

Finally, we can generalize the maximum principle~\ref{lemMaxHyp}.

\begin{lemma}[Generalized Maximum Principle in the hyperbolic 
plane]\label{lemMaxHypGen}
 Let $G$ be a finite graph and let $\alpha:E(G)\to (0,\pi)$ be an admissible
labelling. Let ${\mathscr C}$ be a circle pattern for
$G$ and $\alpha$ which is contained in $\D$. Let ${\mathscr C}^*$ be a 
generalized circle pattern for $G$ and $\alpha$ with exterior intersection 
angles $\beta: V_{\partial,\beta}(G)\to[0,\pi)$. Suppose that for every 
boundary vertex either $v\in V_{\partial,\beta}(G)$ or  the inequality 
$r_{\text{hyp}}^*(v)\geq r_{\text{hyp}}(v)$ holds. Then for all interior 
vertices $r_{\text{hyp}}^*(v)\geq r_{\text{hyp}}(v)$.
\end{lemma}

\begin{remark}
If $V_{\partial,\beta}(G)\not=\emptyset$ or if $r_{\text{hyp}}^*(v)>
r_{\text{hyp}}(v)$ for one boundary vertex then the inequality is strict for 
all interior vertices.
\end{remark}

\begin{proof}
We suppose that $V_{\partial,\beta}(G)\not=\emptyset$ because the case
$V_{\partial,\beta}(G)=\emptyset$ has be proven in Lemma~\ref{lemMaxHyp}. The 
idea of the proof is the same.

Let $v_1,\dots,v_n$ be a numbering of the interior vertices of $G$ where 
$n=|V_{int}(G)|$.
Let $\rho$ and $\rho^*$ be the associated functions to ${\mathscr C}$ and 
${\mathscr C}^*$ respectively as above.
Recall that the circle pattern 
${\mathscr C}$ is the unique minimizer of the strictly convex functional 
$S_{\text{hyp}}$ fixing the boundary values $\rho|_{V_\partial}$.
Consider in $\R^{n}$ the $n$-dimensional interval $U=\{ u\in \R^{n}\ :\ 
u_i<\rho^*(v_i)\}$. 
We will again show that for fixed boundary values $\rho|_{\partial V}$ the 
gradient of $S_{\text{hyp}}(\cdot, \rho|_{\partial V}):\R^n\to\R$ is pointing 
into the complement of $U$ on the boundary $\partial U$ or is contained in 
$\partial U$ (or in one of the hyperplanes $H_i=\{u\in\R^n\ |\ 
u_i= \rho^*(v_i)\}$).

From the proof of Lemma~\ref{lemMaxHyp} we know that if the vertex $v_i$ is not 
incident to any boundary vertex in $V_{\partial,\beta}(G)$ we have
\begin{equation}\label{eqVergleich}
 \frac{\partial S_{\text{hyp}}(\rho)}{\partial u_i}(u_1,\dots, 
u_i=\rho^*(v_i),\dots u_n) 
\geq \frac{\partial S_{\text{hyp}}(\rho)}{\partial u_i} (\rho^*(v_1),\dots, 
\rho^*(v_n))=0,
\end{equation}
where the sum is taken over all edges $[v_i,v_k]\in E(G)$ incident to $v_i$.

Note that 
\begin{equation*}
\frac{\partial \varphi^g_{e}}{\partial \beta} = 
-\frac{\sin\theta \sinh x \sin\beta}{|\cosh(x+i\beta)-\cos\theta|^2}>0
\end{equation*} 
for $\beta>0$ and $x<0$ (and $=0$ for $\beta=0$). Recall that $\frac{\partial 
\varphi_{e}}{\partial u} (x,u,\theta) >0$, therefore we deduce that
\begin{equation*}
\varphi^g_{e}(x,\beta,\theta)\geq 
\varphi^g_{e}(x,0,\theta)=\lim_{u\nearrow 0} 
\varphi_{e}(x,u,\theta) >\varphi_{e}(x,u,\theta)
\end{equation*}
for all $x, u\in\R$, $x,u<0$. This implies as in the proof of 
Lemma~\ref{lemMaxHyp} that estimate~\eqref{eqVergleich} holds for all 
vertices.

Thus the minimum is contained in $U$.
  \end{proof}

As an application we can deduce Theorem~\ref{theoRidhyp} with the same proof as
for Theorem~1.2 in~\cite{He99}.

\section{Rigidity of infinite Euclidean circle patterns}\label{secRigEuc}

In this section we show Theorem~\ref{theoRidEuc} on rigidity of infinite 
circle patterns in the plane. The proof adapts ideas and methods of He's 
rigidity proof in~\cite{He99}. It mainly relies
on the maximum principle in Lemma~\ref{lemMaxHypGen} and on estimates on 
vertex extremal length for patterns of circles explained in 
Section~\ref{secVEL}.

Let $\mathscr D$, $\alpha$, $\mathscr C$ and $\widetilde{\mathscr C}$ be as 
in Theorem~\ref{theoRidEuc}. Let $G$ be the graph of white vertices associated 
to $\mathscr D$. Let $r,\tilde{r}:V(G)\to (0,\infty)$ be the Euclidean radius 
functions of $\mathscr C$ and $\widetilde{\mathscr C}$ respectively.

 The following lemma will be proved later in Appendix~\ref{secProofLemma}.

\begin{lemma}\label{lemCompRad}
Under the assumptions of Theorem~\ref{theoRidEuc} there is a constant $C\geq 1$ 
such that for every vertex $v\in V(G)$
\begin{equation}\label{eqCompRad}
 \frac{1}{C}\leq \frac{\tilde{r}(v)}{r(v)} \leq C.
\end{equation}
\end{lemma}

Assuming this lemma, we will first construct a one-parameter family of immersed 
circle patterns for $\mathscr D$ and $\alpha$ joining $\mathscr C$ and 
$\widetilde{\mathscr C}$ as follows.
For every vertex $v\in V(G)$ denote $\lambda(v)=\log(\tilde{r}(v)/r(v))$ and 
define 
$\hat{r}(v,t):= \text{e}^{\lambda(v)t}r(v)$. Then $\hat{r}(v,0)=r(v)$, 
$\hat{r}(v,1)= \tilde{r}(v)$ and 
\begin{equation}\label{eqEstrtilde}
 \left|\log\frac{\hat{r}(v,t+\eps)}{\hat{r}(v,t)}\right|= |\eps| 
|\lambda(v)| \leq |\eps| \log C
\end{equation}
for $t,t+\eps\in[0,1]$ by~\eqref{eqCompRad}.

Let $G_n$ be an increasing sequence of subgraphs of $G$ whose union is $G$. We 
assume that $G_n$ is the associated graph of white vertices of a b-quad-graph 
${\mathscr D}_n$ which is a cell decomposition of a closed topological disc. 
Then a subpattern of $\mathscr C$ is a circle pattern for $G_n$ and 
$\alpha|_{G_n}$. By Theorem~\ref{theoDirichlet} and 
Proposition~\ref{propRadius} 
for every $n$ and every 
$t\in[0,1]$ there exists an immersed circle pattern ${\mathscr C}_{n,t}$ for 
$G_n$ and $\alpha|_{G_n}$ with boundary values $r_{n,t}|_{\partial V(G_n)} = 
\hat{r}(\cdot, t)|_{\partial V(G_n)}$. Lemma~\ref{lemMaxEuc} together with 
estimate~\eqref{eqEstrtilde} implies that 
\begin{equation*}
 \left|\log\frac{r_{n,t+\eps}(v)}{r_{n,t}(v)}\right| \leq |\eps| \log C
\end{equation*}
for any vertex $v\in V(G_n)$. Thus
\begin{equation}
 \left|\frac{d}{dt}\log r_{n,t}(v)\right| \leq \log C.
\end{equation}

Now, replacing by a suitable subsequence and applying Euclidean transformations 
if necessary, we may assume that for every vertex $v\in V(G)$ the sequence of 
circles ${\mathscr C}_{n,t}(v)$ converges in the Hausdorff metric to some limit 
circle ${\mathscr C}_{t}(v)$. Then ${\mathscr C}_{t}$ is for every $t\in[0,1]$ 
an immersed circle pattern for $G$ and $\alpha$.
By the uniqueness part of Theorem~\ref{theoDirichlet} we may assume that 
${\mathscr C}_{n,0}$ are subpatterns of ${\mathscr C}$ and therefore ${\mathscr 
C}_{0}= \mathscr C$. Similarly we assume that ${\mathscr C}_{1}= {\mathscr 
C}^*$.

From estimate~\eqref{eqEstrtilde} we deduce
\begin{equation}\label{eqEstlogr}
 \left|\frac{d}{dt}\log r_{t}(v)\right| \leq \log C
\end{equation}
where the derivative with respect to $t$ is taken in the generalized 
(distributional) sense.
Let $\ell(v,t)= \log r_{t}(v)- \log r_{0}(v)$. Then we have $\ell(v,0)=0$ and 
\begin{equation}\label{eqBoundh}
\left|\frac{d}{dt}\ell(v,t)\right| \leq \log C
\end{equation} 
by~\eqref{eqEstlogr}.
Let $t\in[0,1]$ be fixed. Define 
\[h(v)=h(v,t) = \frac{d}{dt}\ell(v,t).\]
We will show that $h$ is harmonic in the electrical network based on $G$ with 
conductance 
\begin{equation}\label{eqdefmu}
 \mu([v_0,v_1]) = 2 f'_{\alpha_{[v_0,v_1]}}(\log r_{t}(v_1)-\log r_{t}(v_0)) \
\end{equation}
on the edge $[v_0,v_1]$. Note that $\mu([v_0,v_1])= \mu([v_1,v_0])>0$ by 
Lemma~\ref{lemPropf}, so the conductance is well defined and positive on 
$E(G)$. 

\begin{remark}\label{remArea}
The conductance $\mu([v_0,v_1])$ has the following geometrical 
interpretation. Consider for a circle pattern ${\mathscr C}$ the kite 
associated to the two intersecting circles ${\mathscr C}(v_0)$ and ${\mathscr 
C}(v_1)$ as in Figure~\ref{figKite}~(left). 
Then $\mu([v_0,v_1])$ is equal to the ratio of the Euclidean length 
$H([v_0,v_1])$ of the distance between the intersection points by the 
distance $L([v_0,v_1])=|c(v_0)-c(v_1)|$ between the centers of 
circles.
\end{remark}

Let $v_0\in V(G)$ be an arbitrary vertex and let  $v_1, v_2, \dots, 
v_l,v_{l+1}=v_1$ be the chain of neighboring vertices. The {\em Laplacian} 
$\Delta h$ is then defined as
\begin{equation}\label{eqdeflap}
 \Delta h(v_0)= \sum_{k=1}^l \mu([v_0,v_k]) (h(v_k)-h(v_0)).
\end{equation}

\begin{lemma}\label{lemharm}
 The function $h(\cdot)= h(\cdot,t)$ is harmonic in the network based on $G$ 
where the conductance
is defined in~\eqref{eqdefmu}.
\end{lemma}
\begin{proof}
Remember that we have $\sum_{k=1}^l 2 f_{\alpha_{[v_0,v_k]}}(\log 
r_{t}(v_k)-\log r_{t}(v_0)) -2\pi =0$ by~\eqref{eqFgen} in 
Proposition~\ref{propRadius}. 
Differentiating with respect to $t$ we easily obtain $\Delta h(v_0)=0$ by the 
chain rule.
  \end{proof}

\begin{lemma}\label{lemrec}
For every embedded locally finite circle pattern covering the whole complex 
plane the network based on $G$ with conductance $\mu(e)$ on edges defined 
in~\eqref{eqdefmu} is recurrent.
\end{lemma}

Our proof is based on the following characterization.
\begin{lemma}[{see~\cite[Lemma~9.22]{LP}} or~\cite{GGN13}]\label{lemReff}
 An infinite graph $G$ with edge weights $\mu(e)>0$ is recurrent if for some 
vertex 
$v_0$ there exists a constant $C>0$ such that for every integer $m\geq 0$ there 
exists a finite vertex set $B$ such that 
\begin{equation}
 R_{\text{eff}}(B(v_0,m),V\setminus B)\geq C.
\end{equation}
Here $B(v_0,m)=\{v\in V\ :\ \mathrm{dist}_G(v_0,v)\leq m\}$ is the ball of 
radius $m$ in the graph distance metric, that is the set of vertices which can 
be reached from $v_0$ via a connected path in $G$ with at most $m$ edges. 

The {\em effective resistance} between two sets $A,Z\subset V$ of vertices can 
be defined by the discrete Dirichlet principle
\begin{equation}\label{eqdefReff}
 \frac{1}{R_{\text{eff}}(A,Z)}= \min\left\{{\mathcal E}(g)\ :\ g:V\to\R, 
g|_A=0, 
g|_Z=1\right\}
\end{equation}
where ${\mathcal E}(g)= \sum_{e=[x,y]\in E} 
\mu(e)(g(x)-g(y))^2$ is the discrete {\em Dirichlet energy}.
\end{lemma}

\begin{proof}[Proof of Lemma~\ref{lemrec}.]
 Choose one vertex $v_0\in V$. Without loss of generality we assume that 
$c(v_0)=0$ is the origin.
Let $m\geq 0$ be an integer. Consider all kites in the circle pattern incident 
to one of the vertices $v\in B(v_0,m)$. Let $R_0>0$ be big enough such that the 
open disc $\B(R_0)$ with radius $R_0$ about $0$ covers all such kites. Let 
$A\subset V$ be the set of all vertices whose corresponding centers of circles 
lie inside the disc $\B(R_0)$. 
Now consider all kites which have a non-empty 
intersection with $\B(R_0)$. Let $R_1\geq 2 R_0$ be big enough such that no 
kite intersects both circles $\ci(R_0)$ and $\ci(R_1)$ with radius $R_0$ 
and $R_1$ respectively about $0$. This choice is possible as the circle pattern 
is locally finite. Without loss of generality we may assume that $R_1-R_0=1$ as 
the conductances $\mu(e)$ do not depend on the scaling of the pattern.

Now define the vertex set $B\subset V$ to contain all vertices such that the 
corresponding centers of circles lie inside the open disc $\B(R_1)$. Define 
a function $g:V\to\R$ as follows: $g(v)=0$ if $v\in V\setminus B$ and $g(v)=1$ 
if $v\in A\supset B(v_0,m)$. For $v\in B\setminus A$ let 
$g(v)=d(c(v),\ci(R_1))$ be the Euclidean distance between the center of 
circle $c(v)$ and the circle $\ci(R_1)$. Note that $g:V\to[0,1]$ and for any 
edge $e=[x,y]\in E$ we have $|g(x)-g(y)|\leq L(e)$ by construction.

The effective resistance $R_{\text{eff}}(B(v_0,m),V\setminus B)$ can now easily 
be bounded using the Dirichlet energy ${\mathcal E}(g)\geq 
1/R_{\text{eff}}(B(v_0,m),V\setminus B)$.
\begin{align*}
 {\mathcal E}(g) &= \sum_{e=[x,y]\in E} \mu(e)(g(x)-g(y))^2 \\
 &= \sum_{e=[x,y]\in E_{int}(B)} \mu(e)(g(x)-g(y))^2 +
\sum_{e=[x,y]\in E_{\partial}(B)} \mu(e)(g(x)-g(y))^2 \\
&\leq \sum_{e=[x,y]\in E_{int}(B)} \underbrace{\frac{H(e)}{L(e)} 
L(e)^2}_{=H(e)L(e)= \text{area of the kite } K_e}
+\sum_{e=[x,y]\in E_{\partial}(B)} \mu(e)(g(x)-g(y))^2
\end{align*}
Note that it is sufficient to consider the edges $E_{int}(B)=\{e=[x,y]\in E\ :\ 
x,y\in B\}$ and $E_{\partial}(B)=\{e=[x,y]\in E\ :\ x\in B, y\not\in B\}$. 
Furthermore, for $e=[x,y]\in E_{\partial}(B)$ the term $\mu(e)(g(x)-g(y))^2$ is 
the area of the kite $K_e$ scaled by the factor 
$|g(x)-g(y)|/L(e)$. Now by construction all these kites are contained in the 
disc $\B(3)$ with radius $3$ therefore we finally obtain ${\mathcal E}(g) \leq 
9\pi=C$. This finishes the proof by applying Lemma~\ref{lemReff}.
  \end{proof}

The recurrence of the network implies that the function $h$ is constant. We 
only cite this well known result, see for example~\cite{Woe00}.

\begin{lemma}\label{lemhConst}
Let $G$ be an (infinite) graph with edge weights $\mu:E\to(0,\infty)$ which is 
recurrent. Then any bounded harmonic function $h:V\to\R$ is constant.
\end{lemma}

\begin{proof}[Proof of Theorem~\ref{theoRidEuc}, assuming 
Lemma~\ref{lemCompRad}]
Fix $t\in[0,1]$ and consider the graph $G$ with corresponding edge weights 
$\mu$. Then by Lemma~\ref{lemrec} this graph $G$ is recurrent. 
Lemma~\ref{lemharm} says that $h(v,t)$ is harmonic and also bounded due 
to~\eqref{eqBoundh}. Lemma~\ref{lemhConst} implies that $h(v,t)$ is independent 
of $v$. Therefore $\int_0^1 h(v,s)ds =\ell(v,1)= \log (\tilde{r}(v)/r(v))$ is 
also independent of $v$. This implies that $\widetilde{\mathscr C}$ and 
$\mathscr C$ are images of each other by Euclidean similarities.
   \end{proof}

\section{Some estimates for circle patterns}\label{secVEL}

The main aim of this section is to derive estimates which will be 
useful for the proof of Lemma~\ref{lemCompRad} in the Appendix.
In particular, 
these estimates hold for circle patterns satisfying some geometric 
conditions whose details will be given below.

For motivation we start with the following definitions.
Let $G$ be a graph and $\eta:V\to[0,\infty)$ a function on its vertices. Denote 
the {\em area} and the {\em length} of $\eta$ by
\begin{align*}
 &\text{area}(\eta)=\sum_{v\in V} \eta(v)^2 \qquad \text{and}\qquad
\text{length}_\gamma(\eta)= \int_\gamma d\eta =\sum_{v\in\gamma}\eta(v),
\end{align*}
where $\gamma$ is a subset of the vertices $V$ called {\em curve in $G$}.
If $\Gamma$ is a collection of curves in $G$, then $\eta$ is called {\em 
$\Gamma$-admissible} if $\int_\gamma d\eta\geq 1$ holds for all curves 
$\gamma\in\Gamma$. Then the {\em vertex modulus $\m(\Gamma)$} of $\Gamma$ and 
the {\em vertex extremal length $\vel(\Gamma)$} of $\Gamma$ are defined as
\begin{align*}
 \m(\Gamma) &=\inf\{\text{area}(\eta) : \eta:V\to[0,\infty) \text{ is }
\Gamma-\text{admissible}\} \\
\vel(\Gamma) &=1/\m(\Gamma). 
\end{align*}

Now consider a circle pattern $\mathscr C$ for some planar graph $G$ and dual 
graph $G^*$. We can easily construct admissible functions $\eta$ 
for paths between vertex sets by using the diameter of the circles $\mathscr 
C(v)$.
In order to bound the area, we will restrict ourselves to special classes which 
allow us to relate the 
sum of squared diameters to the area covered by the pattern. 
Note that we are interested in conditions which do not only hold for a single 
finite circle pattern, but either for an infinite circle pattern or for a 
growing sequence of finite circle patterns. 

The main idea is to associate to every vertex $v$ of $G$ not only the circle 
$\mathscr C(v)$ and the closed disc $B(v)$ bounded by it, but also some (simply 
connected) set $S(v)\subset B(v)$ and some closed disc $I(v)\subset S(v)$.

We now explain the three conditions used for 
our proof and also show that the circle patterns considered in 
Theorem~\ref{theoRidEuc} satisfy these properties. Note that we are not 
interested in determining the best possible constants; any suitable value 
suffices.

\begin{description}[\setleftmargin{0pt}\setlabelphantom{Con}]
 \item[{\bf Condition (1)}]
Every point $p\in\C$ which is contained in some kite of the associated kite 
pattern $\mathscr K$ is also contained in 
some $S(w_0)$ for a suitable $w_0\in V(G)$, but only contained in the interior 
of $S(w)$ for a finite, uniformly bounded number of vertices $w\in V(G)$. The 
bound is denoted by $N$.
\item[{\bf Condition (2)}] The discs $I(v)$ have disjoint interiors.
Their radii $r_I(v)$ are uniformly comparable with the radii 
$r(v)$ of the original discs $B(v)$. In particular a uniform estimate holds for 
all interior vertices $v\in V_{\text{int}}$:
\begin{align}\label{estIr}
 r(v)\leq C_0 r_I(v)
\end{align}
for some constant $C_0\geq 1$.
\end{description}
The constants $N$ and $C_0$ may possibly depend on the graph $G$ and on 
other given parameters like the intersection angles. We will specify the 
dependences which are important for our considerations.

In this article, we will mainly use the following definitions for $S(v)$ and 
$I(v)$. Unless otherwise stated, 
for a given circle pattern define for every vertex $v\in V$ the (closed) set 
$D(v)$ as the union of all kites $K_e$ corresponding to edges $e$ incident to 
$v$. Then set $S(v)= B(v)\cap D(v)$. By construction, the union of $S(v)$ for 
$v\in V$ covers every kite of the pattern ${\mathscr K}$ and each interior 
point of a kite is covered 
at most twice by the sets $S(v)$. In the following, we will always use this 
definition of $S(v)$ (unless otherwise stated), so $N=2$ will not depend on $G$ 
or any parameters of the circle pattern.

\begin{lemma}\label{lemEstDisc}
Let $\mathscr C$ be a circle pattern for a graph $G$ and an admissible 
labelling $\alpha:E\to[\alpha_0,\pi)$ with $0<\alpha_0\leq \pi/2$.
Then for every 
interior vertex $v$ there exists a disc $I(v)\subseteq S(v)$ whose radius 
$r_I(v)>0$ is comparable to $r(v)$. In particular, there is a constant $C_0\geq 
1$, possibly depending on $\alpha_0$, such that~\eqref{estIr} holds.
Furthermore, the discs can be chosen such that their interiors are all disjoint.
\end{lemma}
\begin{proof}
 Consider an interior vertex $v$. If $S(v)=B(v)$ the claim is obvious. 
Therefore we may assume that $S(v)\not=B(v)$. Note that
the center $c(v)$ is an interior point of $S(v)$. 
Also, for any kite incident to $c(v)$, the distance of 
$c(v)$ to any point on the boundary $\partial S(v)$ is at least 
$r(v)\sin \alpha_{0}$. Therefore we can take a disc $I(v)$ centered at $c(v)$ 
with 
radius $r(v)\sin \alpha_{0}/2$. This disc is contained in $S(v)$ and the 
interiors of discs for different vertices do not overlap by construction. In 
particular, we may take $C_0=2/\sin \alpha_{0}\geq 1$.
  \end{proof}

In the following, we always consider the discs $I(v)$ defined in the proof of 
the preceding lemma, unless otherwise stated.

\begin{remark}\label{remdefStriang}
Conditions~(1) and~(2) are also satisfied in the following case. 

Consider a triangulation of (a part of) the plane whose triangles have straight 
edges and angles which are uniformly bounded from above and below, i.e.\ in 
$[\delta,\pi-2\delta]$ for some $\delta>0$. 
Let $G^*$ be the (abstract) planar graph corresponding to the triangulation. 
Then the circumcircles of the 
triangles form a circle pattern (with dual graph $G^*$). Let $B_T(v)$ be the 
discs filling the 
circumcircles. Take $S_T(v)\subseteq B_T(v)$ to be the set filling the 
corresponding triangle 
and $I_T(v)$ to be the disc bounded by the incircle of the triangle. Then the 
uniform bounds on the angles of the triangles imply~\eqref{estIr} for some 
constant $C_0\geq 1$. Also, the discs $I_T(v)$ obviously have disjoint 
interiors and condition~(1) is also satisfied by construction (with $N=2$ not 
depending on any parameters).
\end{remark}

In the following, we will use the Euclidean area of a (measurable) 
domain $D\subset\C\cong \R^2$, denoted by $\textsc{Area}(D)$.

\begin{corollary}\label{corEstDiscs} 
Let $G$ be a planar graph with associated pattern of circles which 
satisfies conditions~(1) and~(2). Let $D$ be a bounded subset of $\C$ and let
$V_D$ be the subset of vertices whose discs ${B}(v)$ are completely contained 
in $D$. Then we have 
\[\sum_{v\in V_D} r(v)^2 \leq \frac{C_0^2}{\pi} \textsc{ 
Area}(D).\]
In particular, this estimate holds for circle patterns for a graph $G$ and an 
admissible labelling $\alpha:E\to[\alpha_0,\pi)$ with $\alpha_0>0$.
\end{corollary}

\paragraph{{\bf Condition (3)}:}
This condition gives a special length estimate. Let $\ci(R)=\ci_c(R)$ be a 
circle with center $c$ and Radius $R$ and denote $\B(R)$ the disc bounded by 
$\ci(R)$. Let $v\in V$ be such that $\B(R)\cap (\text{interior of } 
S(v))\not=\emptyset$. Assume that
the area of $\B(R)\cap B(v)$ is at most $1/5$ times the area of $B(v)$.
Denote by $l(v)$ the length of the arcs $\ci(R)\cap 
S(v)$ contained in $S(v)$ and by $d(v)$ the maximum of the  distance of a point 
in $S(v)\cap \B(R)$ to the arc $\ci(R)\cap S(v)$. We want $d(v)$ to be  
uniformly bounded by the lengths $l(v)$ up to some constant, in particular 
there should exist a 
constant $C_1\geq 1$, independent of $R$, $c$ and $v$, such that
\begin{align}\label{estdl}
 d(v)\leq C_1 l(v).
\end{align}

We will now prove this condition for circle patterns.

\begin{lemma}
 Given a circle pattern for a graph $G$ and an admissible labelling 
$\alpha:E\to[\alpha_0,\pi)$ with $0<\alpha_0\leq \pi/2$.
Then there exists a constant $C_1\geq 1$ such that~\eqref{estdl} holds.
\end{lemma}
\begin{proof}
Let $\ci(R)$ be a circle and let $v\in V$ be such that $\B(R)\cap 
(\text{interior of } S(v))\not=\emptyset$. Assume that
the area of $\B(R)\cap B(v)$ is at most $1/5$ times the area of $B(v)$. Then 
either $\B(R)\subset B(v)$ or $\B(R)$ cannot contain the center 
$c(v)$ of the disc $B(v)$ (by geometric considerations, see also below).

We will show the existence of a constant $C_1\geq 1$ such that~\eqref{estdl} 
holds by considering all possible cases and determining a suitable constant 
$C_1$ in each case. The final constant can be taken as the maximum of all these 
constants.

First note, that~\eqref{estdl} holds with $C_1=1$ if the whole arc $\ci(R)\cap 
B(v)$ is contained in $S(v)$.

 By our assumptions we know that $S(v)$
contains the center $c(v)$ as an interior point. Also $d(v)\leq r(v)$ as the 
area of $\B(R)\cap B(v)$ is at most $1/5$ times the area of $B(v)$. In fact, we 
need more precise consequences of this condition.
If $\B(R)\subset B(v)$ we deduce $d(v)\leq R\leq r(v)/\sqrt{5}$.
For the remaining case $\B(R)\setminus B(v)\not=\emptyset$ recall that 
the area of the lunar region $\B(R)\cap B(v)$ can be calculated via
\begin{align*}
 \textsc{Area}(\B(R)\cap B(v)) = & r^2\arccos\left( 
\frac{d_c^2+r^2-R^2}{2rd_c}\right) 
+R^2\arccos\left( \frac{d_c^2+R^2-r^2}{2Rd_c}\right) \\
&-\frac{1}{2}\sqrt{(r+R-d_c)(r+R+d_c)(d_c+r-R)(d_c-r+R)},
\end{align*}
where $r=r(v)$ and $d_c\geq r$ is the distance of the centers of $\B(R)$ and 
$B(v)$, see Figure~\ref{figLune}. 
\begin{figure}[tb]
\labellist
\small\hair 2pt
 \pinlabel {$r$} [ ] at 69 109
 \pinlabel {$R$} [ ] at 140 110
 \pinlabel {$d_c$} [ ] at 128 75
 \pinlabel {$\gamma$} [ ] at 71 92
\endlabellist
\centering
\includegraphics[scale=0.5]{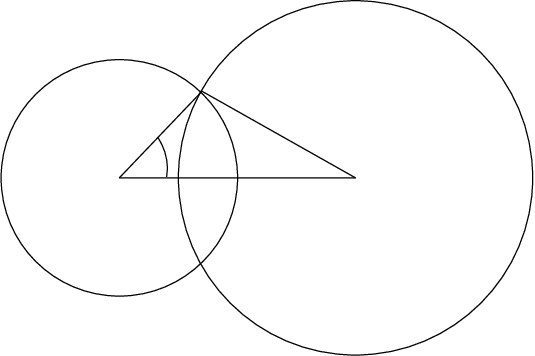}
\caption{The lunar region $\B(R)\cap B(v)$ and the angle $\gamma= 
\arccos(\frac{d_c^2+r^2-R^2}{2rd_c})$.}
\label{figLune}
\end{figure}
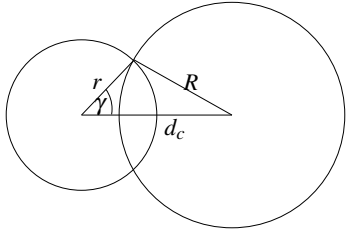
Now we deduce from $\textsc{Area}(\B(R)\cap B(v))\leq \pi r^2/5$ by 
elementary estimates that $d_c> R$ and in particular $R+r-d_c\leq C_2 r$ for 
some absolute constant $0<C_2<1$. Moreover, the area estimate
implies that the angle at $c(v)$ of the circular 
sector of $B(v)$ containing the lunar region $\B(R)\cap B(v)$ is bounded by 
$2\gamma_{max}<\pi$, that is
$0<\arccos(\frac{d_c^2+r^2-R^2}{2rd_c})\leq \gamma_{max}<\pi/2$ for some 
constant $\gamma_{max}$.

In the following, we will frequently use
that for any two points on the circular arc $\ci(R)\cap B(v)$ the length of the 
straight line segment connecting these points is larger than the distance of 
any point on this segment to this arc of $\ci(R)$.

We will distinguish several cases. First, if the lune $\B(R)\cap B(v)$ is 
completely covered by kites of $D(v)$, i.e.\ $\B(R)\cap B(v)=\B(R)\cap S(v)$, 
then the claim holds by simple geometric reasoning as mentioned above.
So we will restrict to the 
case when $\ci(R)\cap (B(v)\setminus S(v))\not=\emptyset$. In this case there 
is an intersection angle $\alpha_e<\pi/2$ for an edge $e$ incident to $v$, so 
in particular $0<\alpha_0<\pi/2$. Denote by $\B_{\text{int}}(R)$ the interior 
of the disc $\B(R)$. 
For each connected component of $\B(R)\cap S(v)$ we define $l$ to be the 
length of the corresponding arc of $\ci(R)\cap 
S(v)$ and $d$ the maximum of the  distance of a point in this component of 
$S(v)\cap \B(R)$ to the corresponding arc of $\ci(R)\cap S(v)$.
If the estimate $d\leq \tilde{C}_1 l$ holds for each of these components with 
suitable constants $\tilde{C}_1$ (which depend only on $\alpha_0$, $C_2$, 
$\gamma_{max}$ and the global assumptions), we take the largest of all these 
constants for $C_1$ and~\eqref{estdl} holds.
\begin{enumerate}[(i)]
 \item We start with the case when the circular arc ${\mathscr C}(v)\cap 
\B_{\text{int}}(R)$ does not contain any black vertex. The 
possible configurations are illustrated in Figure~\ref{FigCasei}.
The estimate $l\geq d\sin\alpha_0$ holds as $\beta\geq \alpha_0$ or $\beta_1$ 
and $\beta_2$  are at least $\pi/2$ 
(see Figure~\ref{FigCasei}).
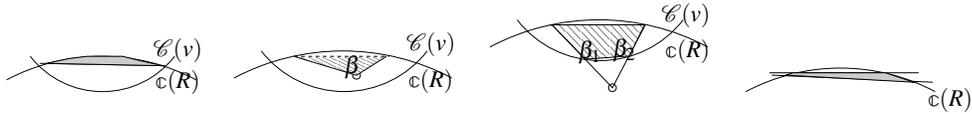
\begin{figure}[h]
\begin{center}
\labellist
\small\hair 2pt
  \pinlabel {${\mathscr C}(v)$} [ ] at 206 50
 \pinlabel {$\ci(R)$} [ ] at 206 12
\endlabellist
 \includegraphics[width=0.2\textwidth]{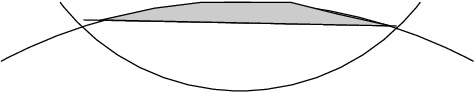}
\hspace{2ex}
\labellist
\small\hair 2pt
 \pinlabel {$\beta$} [ ] at 124 26
 \pinlabel {${\mathscr C}(v)$} [ ] at 206 50
 \pinlabel {$\ci(R)$} [ ] at 206 12
\endlabellist
 \includegraphics[width=0.23\textwidth]{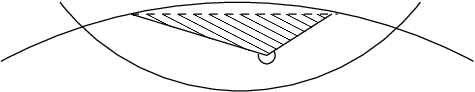}
\hspace{2ex}
\labellist
\small\hair 2pt
 \pinlabel {$\beta_1$} [ ] at 105 42
 \pinlabel {$\beta_2$} [ ] at 141 44
 \pinlabel {$\ci(R)$} [ ] at 207 43
 \pinlabel {${\mathscr C}(v)$} [ ] at 205 80
\endlabellist
\includegraphics[width=0.23\textwidth]{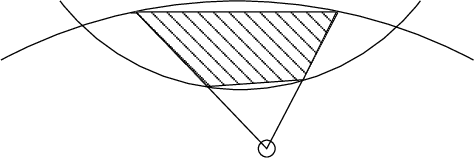}
\hspace{2ex}
\labellist
\small\hair 2pt
 \pinlabel {$\ci(R)$} [ ] at 246 -10
\endlabellist
 \includegraphics[width=0.2\textwidth]{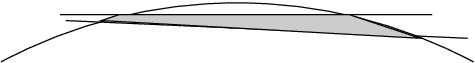}
\end{center}
\caption{Illustration of the possible cases where the considered 
component of $S(v)\cap \B_{int}(R)$ does not contain a black 
vertex.}\label{FigCasei}
\end{figure}
\item Assume now that the circular arc of ${\mathscr C}(v)\cap 
\B_{\text{int}}(R)$ corresponding to the connected component in consideration 
contains exactly one black vertex.
Figure~\ref{FigCaseii} illustrates schematically the possible 
configurations. The two intersection angles 
satisfy $\alpha_1,\alpha_2\in[\alpha_0,\pi/2)$. In the first case, we deduce 
from the sine law 
that $l\geq d\sin\alpha_0$. In the second and third case, consider the 
kite which 
has one white vertex in $S(v)$. Geometric considerations show that $2l\geq 
d\sin\alpha_2$ holds if $\alpha_2\leq \pi/4$ or if $\gamma\geq \pi/2$. In 
the remaining case we obtain $l\geq d(\sin\alpha_2)^2$.
\begin{figure}[h]
\begin{center}
 \resizebox{0.29\textwidth}{!}{\input{Caseii1.pspdftex}}
\hspace{0.02\textwidth}
 \resizebox{0.29\textwidth}{!}{\input{Caseii2.pspdftex}}
\hspace{0.02\textwidth}
 \resizebox{0.29\textwidth}{!}{\input{Caseii3.pspdftex}}
\end{center}
\caption{Illustration of the three possible cases where ${\mathscr C}(v)\cap 
\B_{int}(R)$ contains only one black vertex of the considered 
component.}\label{FigCaseii}
\end{figure}
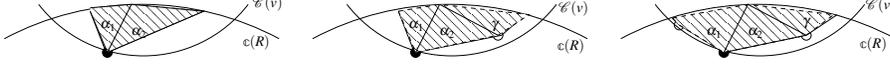
\item The last case consists of those configurations where $\B(R)\setminus 
B(v)\not=\emptyset$ and the connected 
component of $S(v)\cap \B(R)$ in consideration contains at least one 
black vertex in the circular arc ${\mathscr C}(v)\cap \B_{\text{int}}(R)$ 
and another intersection point in ${\mathscr C}(v)\cap 
\B_{\text{int}}(R)$. Consider the smallest circular sector $B_{\text{sec}}(v)$ 
of $B(v)$ 
containing the considered component of $S(v)\cap {\mathscr C}(v)\cap 
\B_{\text{int}}(R)$. Denote by $2\beta>0$ its 
central angle, see Figure~\ref{figSec}, and by $2\beta_0\geq 2\beta$ the 
central angle of the circular 
sector $B(v)$ containing the whole arc ${\mathscr C}(v)\cap 
\B_{\text{int}}(R)$. 
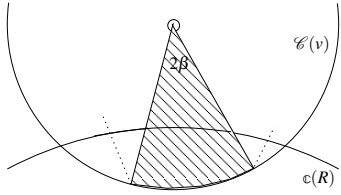
\begin{figure}[htb]
\begin{center}
 \resizebox{0.35\textwidth}{!}{\input{CaseSec.pspdftex}}
\end{center}
\caption{Illustration of the circular sector $B_{\text{sec}}(v)$ 
(shaded) and its central angle $2\beta$.}\label{figSec}
\end{figure}
Note that by our assumptions and reasoning as in the previous cases the angles 
of considered component of $S(v)\setminus B_{\text{sec}}(v)$ at the (at most 
two) points in ${\mathscr C}(v)\cap \B_{\text{int}}(R)$ are at least 
$\alpha_0$. So the estimates of the previous cases apply. 

Note that the length of the intersection of a radius with a lunar region is 
largest for the direction of the line connecting the two centers of circles.
Therefore, in case $\beta\leq \beta_0/2$ it is sufficient 
to consider the longest of the straight edges of the boundary of 
$B_{\text{sec}}(v)\cap \B(R)$. Its length is at least one half of the maximum 
distance 
of  $B_{\text{sec}}(v)\cap {\mathscr C}(v)$ to $B_{sec}(v)\cap \ci(R)$ in 
$B_{\text{sec}}(v)\cap \B(R)$. Consequently, the desired estimate $d\leq C_1l$ 
for some constant $C_1$ follows from the previous cases.

In the remaining case $\beta\geq \beta_0/2$ the distance of 
$B_{\text{sec}}(v)\cap {\mathscr C}(v)$ to $B_{\text{sec}}(v)\cap \ci(R)$ in 
$B_{\text{sec}}(v)\cap \B(R)$ is $R+r(v)-d_c$ where $d_c$ is the distance of 
the 
centers of $\B(R)$ and $B(v)$ as above. Using the fact that the angle $\beta_0$ 
is bounded by $\gamma_{\max}<\pi/2$ from above, it follows by elementary 
calculations and estimates that there exists a constant $C_3>0$ such that 
length of $B_{\text{sec}}(v)\cap \ci(R)$ is at least $C_3 (R+r(v)-d_c)$. 
This implies the desired estimate $d\leq C_1l$ for some constant $C_1$.
\end{enumerate}
  \end{proof}

\begin{remark}
 The same proof as above shows that estimate~\eqref{estdl} also holds in the 
case of a triangulation of (a part of) the plane whose 
triangles have uniformly strictly acute angles, i.e.\ in 
$[\delta,\frac{\pi}{2}-\delta]$ for some $\delta>0$. In this case the
sets $S_T(v)$ are chosen as in Remark~\ref{remdefStriang}.
\end{remark}

\subsection{Estimates on vertex extremal length}\label{secVELest}

In the following we will use the notion of vertex extremal length 
introduced and studied by Cannon and others, see for 
example~\cite{Ca94,HeSch95}. The proofs are mostly based on suitably adapted 
ideas in~\cite{He99}. 
The main aim are Lemmas~\ref{lemEstSumVel}, \ref{lemEstVel}, 
and~\ref{lemIntRad} 
which we need for the proof of Lemma~\ref{lemCompRad} in the Appendix.

We begin with some notation. Let $G$ be a connected planar graph.
Define a {\em path in $G$} to be a sequence of vertices $(v_0,v_1,v_2,\dots, 
v_k,\dots )$ such that $[v_{k-1},v_k]$ is an edge for all $k\geq 1$. We 
identify a path in $G$ with the corresponding curve. Let $\Gamma_G(V_1,V_2)$ be 
the collection of paths from $V_1$ to $V_2$. We allow $V_2=\infty$ by taking 
paths with infinitely many different vertices. Denote by 
$\vel(V_1,V_2)=\vel(\Gamma_G(V_1,V_2))$ the vertex extremal length between 
$V_1$ and $V_2$.

Given three subsets $V_1,V_2,V_3\subset V$ we say that {\em $V_2$ separates 
$V_1$ and $V_3$} if every path from $V_1$ to $V_3$ contains a vertex of $V_2$. 
Note that $V_1\cap V_3\not=\emptyset$ is possible. The next lemma follows by a 
classical argument, see for example~\cite[Chapter~1~D.]{Ah66}. 

\begin{lemma}\label{lemEstSumVel}
 Let $V_1,V_2,\dots, V_{2m}$ be subsets of $V$ which are pairwise disjoint. 
Assume that $V_{i_2}$ separates $V_{i_1}$ and $V_{i_3}$ for $1\leq 
i_1<i_2<i_3$. We allow $V_{2m}=\infty$. Then we have
\[\vel(V_1,V_{2m})\geq \sum_{k=1}^m \vel(V_{2k-1},V_{2k}).\]
\end{lemma}

We will relate vertex extremal length to geometric properties of a given 
(locally finite) planar circle pattern. 

In the following, we associate to a circle $\ci(R)=\ci_c(R)$ with 
center $c$ and radius $R$ the subset of vertices
\[V_{\ci(R)}=\{ v\in V : S(v)\cap \ci(R)\not=\emptyset\}.\]
In the following, we will consider circles $\ci(R)$ for different radii $R$, 
which are all assumed to have the same center $c$.
Our next aim is to obtain a lower bound on the vertex extremal length. This 
estimate holds for all circle patterns which satisfy conditions~(1), (2)
and~(3), in particular for circle patterns considered in 
Theorem~\ref{theoRidEuc}.

\begin{lemma}\label{lemEstVel} 
Let $G$ be a connected graph associated to a locally finite circle pattern 
${\mathscr C}$. Assume that conditions~(1)--(3) hold.
Then there exists a constant $C_2=1/(144C_0^2 +16C_1^2\pi^2)>0$,
such that for all 
$0<R_1<R_2$ there holds: If $V_{\ci(R_2)}\not=\emptyset$, $V_{\ci(R_1)}\cap 
V_{\ci(R_2)}=\emptyset$, and the 
center $c$ is not contained in any of the discs $B(v)$ for $v\in V_{\ci(R_2)}$, 
then
\begin{align*}
 \vel(V_{\ci(R_1)}, V_{\ci(R_2)})\geq 4C_2\frac{(R_2-R_1)^2}{R_2^2}.
\end{align*}
In particular, if $R_2\geq 2R_1$ and $V_{\ci(R_1)}\cap 
V_{\ci(R_2)}=\emptyset$ and the center $c$ is not contained in any of the discs 
$B(v)$ for $v\in V_{\ci(R_2)}$, then
\begin{align*}
 \vel(V_{\ci(R_1)}, V_{\ci(R_2)})\geq C_2.
\end{align*}
\end{lemma}
\begin{proof}
 For $v\in V$ define $d(v):=\min\{2R_2,2r(v)\}$
if the area of the region $B(v)\cap \B(R_2)$ is $\textsc{Area}(B(v)\cap 
\B(R_2))\geq\frac{1}{5}\pi r(v)^2$. If $v\in V_{\ci(R_2)}$ and if
$\textsc{Area}(B(v)\cap \B(R_2))<\frac{1}{5}\pi r(v)^2$ 
then let $d(v)$ be the the maximum of the distance of a point 
in $S(v)\cap \B(R_2)$ to the arc of $\ci(R_2)\cap S(v)$ in the same connected 
component as in condition~(3). Else 
set $d(v)=0$. Then for all paths $\gamma$ from $V_{\ci(R_1)}$ to $V_{\ci(R_2)}$ 
we have 
$\sum_{v\in\gamma} d(v)\geq R_2-R_1$. Therefore, $\eta(v)=d(v)/(R_2-R_1)$ is a 
$\Gamma_G( V_{\ci(R_1)},V_{\ci(R_2)})$-admissible function. Furthermore,
\begin{align*}
 \m(\Gamma_G( V_{\ci(R_1)},V_{\ci(R_2)}))&\leq \text{area}(\eta) \\
&\leq \sum\limits_{\textsc{Area}(B(v)\cap \B(R_2))\geq\frac{1}{5}\pi r(v)^2} 
\frac{\min\{4r(v)^2,4R_2^2\}}{(R_2-R_1)^2} \\
&\qquad + \sum\limits_{\textsc{Area}(B(v)\cap 
\B(R_2))<\frac{1}{5}\pi r(v)^2} \frac{d(v)^2}{(R_2-R_1)^2}.
\end{align*}
Similarly as in Corollary~\ref{corEstDiscs} 
\begin{align*}
\sum\limits_{\textsc{Area}(B(v)\cap \B(R_2))\geq\frac{1}{5}\pi r(v)^2} 
\hspace{-1.3ex}
\frac{\min\{4r(v)^2,4R_2^2\}}{(R_2-R_1)^2} \leq \frac{36}{\pi}C_0^2 
\frac{\textsc{Area}(\B(R_2))}{(R_2-R_1)^2} = 36 C_0^2 
\frac{R_2^2}{(R_2-R_1)^2}.
\end{align*}
Now condition~(3) guarantees that 
\[\sum\limits_{\textsc{Area}(B(v)\cap \B(R_2))<\frac{1}{5}\pi r(v)^2} d(v) \leq 
C_1 \text{length}(\ci(R_2)) =2C_1\pi R_2.\] 
Thus the second sum is bounded by $4C_1^2\pi^2R_2^2/(R_2-R_1)^2$. We may 
therefore take $C_2=1/(144C_0^2 +16C_1^2\pi^2)$.
   \end{proof}

The previous lemma allows to characterize the type of graphs associated to 
locally finite circle pattern satisfying conditions~(1)--(3). 

In general, let $G$ be a connected infinite planar graph and $V_0\subset V$ be 
a non-empty finite subset of 
vertices. Then $G$ is called {\em $\vel$-parabolic}, if 
$\vel(V_0,\infty)=\infty$ and {\em $\vel$-hyperbolic} otherwise. Note that 
these definitions are independent of $V_0$.

\begin{lemma}
 Let $G$ be the planar graph associated to an embedded circle pattern 
satisfying conditions~(1)--(3). If the pattern is locally finite in $\C$, then 
$G$ is $\vel$-parabolic.
\end{lemma}
\begin{proof}
 If the pattern is locally finite, we can find inductively a sequence of 
circles $\ci(R_j)$ with $R_{j+1}\geq 2R_j$ and $V_{\ci(R_{j+1})}\cap 
V_{\ci(R_j)}=\emptyset$ and $V_{\ci(R_1)}\not=\emptyset$. Furthermore, we can 
assume (by suitably enlarging $R_1$ if necessary as the circle pattern is 
locally finite) that the center $c$ is not contained in any of the discs $B(v)$ 
for all $v\in V_{\ci(R_1)}$. Then 
$V_{\ci(R_{i_2})}$ separates $V_{\ci(R_{i_1})}$ and $V_{\ci(R_{i_3})}$ for all 
$1\leq i_1<i_2<i_3$. 
Now Lemmas~\ref{lemEstSumVel} and~\ref{lemEstVel} imply
\begin{align*}
 \vel(V_{\ci(R_1)},\infty)\geq \sum_{k=1}^\infty \vel(V_{\ci(R_{2k-1})}, 
V_{\ci(R_{2k})}) \geq \sum_{k=1}^\infty C_2 =\infty.
\end{align*}
   \end{proof}

\begin{lemma}
 Let $G$ be the planar graph associated to an embedded circle pattern 
satisfying 
conditions~(1)--(3). Let $\mu:E\to[0,\infty)$ be conductances on the edges of 
$G$. If $\sum_{e=[v,w]}\mu(e) \leq C_4$ holds for some constant $C_4>0$ and all 
vertices $v\in V$, then for any disjoint subsets $V_1,V_2\subset V$ there holds
\begin{equation*}
 \vel(V_1,V_2)\leq 2C_4 R_{\text{eff}}(V_1,V_2).
\end{equation*}
In particular, if $G$ is $\vel$-parabolic and $G$ is connected, then $(G,\mu)$ 
is recurrent.
\end{lemma}
The proof is completely analogous to the proof of Lemma~5.4 
in~\cite{He99}. Note that together with the following lemma this provides 
another proof of Lemma~\ref{lemrec}.
 
\begin{lemma}
 Let $\mathscr C$ be an embedded circle pattern for $G$ and 
$\alpha:E(G)\to[\alpha_0,\pi)$ with $0<\alpha_0\leq \pi/2$. Define conductances 
as in~\eqref{eqdefmu}. Then 
there is a constant $C_5=C_5(\alpha_0)$ such that 
\begin{equation*}
\sum_{e=[v,w]}\mu(e) \leq C_5
\end{equation*}
holds for all interior vertices $v\in V$.
\end{lemma}
\begin{proof}
First note that for any kite $K_e$ corresponding to the edge $e=[v,w]$ 
the conductance $\mu(e)$ may be expressed as 
\[\mu(e)=\mu_\alpha(\beta)= \sin(2\beta)+(1-\cos(2\beta))\cot\alpha,\]
where $\alpha=\alpha_e\in[\alpha_0,\pi)$ is the intersection angle and 
$2\beta\in(0,2(\pi-\alpha))$
is the angle of the kite $K_e$ at the vertex $v$. Note that for fixed 
$\alpha$ the function $\mu_\alpha$ has a unique maximum for 
$\beta=(\pi-\alpha)/2$ in its domain with value $\cot(\alpha/2)$.

Let $v\in V_{int}(G)$. Let $e_1,\dots,e_m$ be the edges incident to $v$ and 
$2\beta_1,\dots, 2\beta_m$ the angles of the kites 
$K_{e_1},\dots,K_{e_m}$ at $v$. As the intersection angles 
$\alpha:E(G)\to[\alpha_0,\pi)$ are fixed, we need to bound the maximum of the 
function $F(\beta_1,\dots, \beta_m)=\sum_{i=1}^m \mu_{\alpha_i}(\beta_i)$ where
$\alpha_i=\alpha_{e_i}$ and
$\beta_i\in [0,\pi-\alpha_i]$
under the constraint that $\sum_{i=1}^m\beta_i=\pi$. (This follows from the 
embedding of the incident kites.)

Critical interior points then satisfy $\mu_{\alpha_i}'(\hat{\beta}_i)=\lambda=$ 
constant for all $i=1,\dots,m$, where $\lambda\in[-2\max\{1,1/\sin\alpha_i\}, 
2\max\{1,1/\sin\alpha_i\}]$. Thus, all angles 
$\hat{\beta}_i$ are larger than $(\pi-\alpha_i)/2$ or all smaller. In the  
second case, we can express 
$\mu_{\alpha_i}(\hat{\beta}_i)=(1+\lambda/2)\tan\hat{\beta}_i$. As 
$\hat{\beta}_i\leq (\pi-\alpha_i)/2\leq (\pi-\alpha_0)/2$, we estimate 
$\tan \hat{\beta}_i\leq \hat{\beta}_i \frac{2}{\pi-\alpha_0} 
\cot(\frac{\alpha_0}{2})$ and obtain an 
estimate $F(\hat{\beta}_1,\dots, \hat{\beta}_m)\leq C(\alpha_0)$ for some 
constant $C(\alpha_0)$. 
In case where $\hat{\beta}_i> (\pi-\alpha_i)/2$ holds for all $i=1,\dots,m$, 
there are at most three kites with intersection angles $\alpha_i\leq \pi/2$. 
For all other kites $\mu_{\alpha_i}$ is defined for 
$\beta_i\in[0,\pi-\alpha_i]$. Thus we deduce that 
\begin{multline*}
F(\hat{\beta}_1,\dots, \hat{\beta}_m)\leq 3\cot(\alpha_0/2) 
+\sum_{\alpha_i>\pi/2}\left( \mu_{\alpha_i}(0)+\int_0^{\hat{\beta}_i} 
\mu_{\alpha_i}'({\beta}_i) d\beta_i \right) \\
\leq 3\cot(\alpha_0/2) +2\pi.
\end{multline*}
Now consider the case that some variable $\beta_i$ assumes its minimum or 
maximum. If there is a minimal and a maximal value, say $\beta_i=0$ and 
$\beta_j=\pi-\alpha_j$, then a simple Taylor expansion shows that 
\[ F(\beta_1,\dots,\beta_i,\dots,\beta_j,\dots, \beta_m)< 
F(\beta_1,\dots,\beta_i+\eps,\dots,\beta_j-\eps,\dots, \beta_m)\]
for $0<\eps$ small enough. Thus the maximum of $F(\beta_1,\dots, \beta_m)$ is 
not attained at such a point. As 
$\mu_{\alpha_i}(0)=0=\mu_{\alpha_i}(\pi-\alpha_i)$, for $\hat{\beta}_i=0$ or 
$\hat{\beta}_i=\pi-\alpha_i$ we just consider the function $F$ where the $i$th 
term is missing and in the second case the constraint is changed to 
$\sum_{j=1}^m\beta_j=\alpha_i$. Thus the previous reasoning applies. 
Furthermore, if $\beta_i$ are all minimal or all maximal then $F$ assumes it 
global minimum at this point.
  \end{proof}

Let $G$ be a planar graph and $V_1,V_2\subset V$ be two nonempty subsets of 
vertices. We allow $V_2=\infty$. Let $\Gamma_G^*(V_1,V_2)$ be the collection of 
vertex curves in $G$ which separate $V_1$ and $V_2$. Then by arguments similar 
to the duality argument for the extremal length of curve families in the plane 
(see for example~\cite{Sch93,HeSch95}), there holds
\[\m(\Gamma_G^*(V_1,V_2))= \frac{1}{\m(\Gamma_G(V_1,V_2))} =\vel(V_1,V_2).\]
Note that if $G$ is the 1-skeleton of a strongly regular cell decomposition of 
an open disc and $\mathscr C$ is a corresponding circle pattern, then 
$V_\ci$ is a connected set of vertices for any circle $\ci$ and $\Sp(V_\ci)= 
\bigcup_{v\in V_\ci} S(v)$ is pathwise connected. This is also true for any 
path $\gamma$ in $G$ and the set $\Sp(\gamma)= \bigcup_{v\in \gamma} S(v)$.

\begin{lemma}\label{lemIntRad}
Let $G$ be the 1-skeleton of a strongly regular cell decomposition of 
an open disc. Let $\mathscr C$ be a corresponding circle pattern such
that conditions~(1)--(3) hold.
Let $V_0=\{v_0\},V_1,V_2\subseteq V$ be pairwise disjoint subsets of vertices 
and $V_3=\infty$. Assume that for $0\leq i_1<i_2<i_3\leq 3$ the set $V_{i_2}$ 
separates $V_{i_1}$ and $V_{i_3}$. Assume further that $S(v_0)\cap S(v)=
\emptyset$ for all $v\in V_1$. Then there is a constant 
$C_6= 9/(4C_2)= 9(144C_0^2 +4C_1^2\pi^2)>0$,  such that the following is true. 

If $\vel(V_1,V_2)>C_6$, then there is some $R>0$ such that for any circle 
$\ci(\rho)=\ci_{v_0}(\rho)$ with center $v_0$ and radius $\rho\in[R,2R]$ the 
set $V_{\ci(\rho)}$ separates $V_1$ and $V_2$. 
\end{lemma}
\begin{proof}
 Without loss of generality we can assume that $v_0$ is the origin. This is 
also the center of the discs and circles $\B(\rho)$ and $\ci(\rho)$ for all 
$\rho$ in this proof. Set
\[R=\min\{\hat{R} : S(v)\cap \B(\hat{R})\not=\emptyset \text{ for all } v\in 
V_1\} =\max\{ d(0,S(v)), v\in V_1\}.\]
Then $R>0$ as $S(v_0)\cap S(v)= \emptyset$ for all $v\in V_1$.
Without loss of generality we may also assume $R=1$. As any curve
$\gamma^*\in\Gamma^*_G(V_1,V_2)$ separates $V_1\cup\{v_0\}$ from $\infty$, we 
deduce that the diameter of $\Sp(\gamma^*)$ is at least $R=1$.

Assume by contradiction that the claim fails, i.e.\ no suitable constant $C_6$ 
exists. Then there is $\rho_1\in[1,2]$ such that $V_{\ci(\rho_1)}$ does not 
separate $V_1$ and $V_2$. So there is a path $\gamma_0$ from $V_1$ to $V_2$ 
such that $\gamma_0\cap V_{\ci(\rho_1)}=\emptyset$. This implies in particular 
that $\Sp(\gamma_0)\cap \ci(\rho_1) =\emptyset$, so the 
continuum $\Sp(\gamma_0)$ 
is either contained in $\B(\rho_1)$ or in $\C\setminus \B(\rho_1)$. As 
$\B(1)\cap S(v)\not=\emptyset$ for all $v\in V_1$ by construction and as 
$\gamma_0$ is a path from $V_1$ to $V_2$, we deduce that $\B(\gamma)$ is 
contained in the interior of $\B(\rho_1)$ as $\rho_1\geq 1$.

Now consider a curve $\gamma^*\in\Gamma^*_G(V_1,V_2)$. Then
\[\Sp(\gamma^*)\cap \B(\rho_1)\supset \Sp(\gamma^*)\cap \Sp(\gamma_0)
\supset \Sp(\gamma^*\cap \gamma_0)\not=\emptyset.\]
Note that by our assumptions on $G$ every curve 
$\gamma^*\in\Gamma^*_G(V_1,V_2)$ contains a connected 
subcurve $\hat{\gamma}^*\in\Gamma^*_G(V_1,V_2)$ as $V_1$ and $V_2$ are  
disjoint. Furthermore either $\Sp(\hat{\gamma}^*)\subset 
\B(3)$ or $\Sp(\hat{\gamma}^*)$ is a continuum connecting $\ci(\rho_1)$ and 
$\ci(3)$.
Define the vertex function $\eta(v)= \min\{2r(v),6\}$
if the area of the region $\textsc{Area}(B(v)\cap 
\B(3))\geq\frac{1}{5}\pi r(v)^2$ as in the proof of Lemma~\ref{lemEstVel}. 
Furthermore, if $v\in V_{\ci(3)}$ and  if
$\textsc{Area}(B(v)\cap \B(3))<\frac{1}{5}\pi r(v)^2$ 
then let $\eta(v)$ be the maximum of the length of the arc $\ci(3)\cap S(v)$ 
and of the maximum of the  distance of a point 
in $S(v)\cap \B(3)$ to the arc $\ci(3)\cap S(v)$ as in condition~(3). Else 
set $\eta(v)=0$. Then for any curve 
$\gamma^*\in\Gamma^*_G(V_1,V_2)$ we have $\sum_{v\in\gamma^*} \eta(v)\geq 1$ as 
$1+\rho_1\leq 3$. Thus $\eta$ is $\Gamma^*_G(V_1,V_2)$-admissible. Moreover, by 
similar reasoning as in the proof of Lemma~\ref{lemEstVel} we obtain that
\[\m(\Gamma^*_G(V_1,V_2))\leq \text{area}(\eta)\leq (144C_0^2 +4C_1^2\pi^2) 
\cdot 9=:C_6.\] 
This contradicts our assumption and finishes the proof.
  \end{proof}

With the same proof as in~\cite[Corollary~6.2]{He99} (adapting the 
numerical constant to $C_6$) we obtain

\begin{corollary}
 Let $G$ be the 1-skeleton of a strongly regular cell decomposition of 
an open disc and $\mathscr C$ is a corresponding circle pattern. Assume 
that conditions~(1)--(3) hold. Then $\mathscr C$ is locally finite in $\C$ if 
and only if $G$ is $\vel$-parabolic.
\end{corollary}

\subsection{Topological properties of discrete conformal maps on circle 
patterns}\label{secTopProp}
In this section we consider two circle patterns with the same combinatorics
which are contained in two given simply connected sets $D,\widetilde{D}\subset
\C$ respectively as in Theorem~\ref{theoConv}.
Our main aim is to
study topological properties of such a sequence of pairs of patterns and the 
sequence of the corresponding discrete conformal maps from 
Definition~\ref{deffdiamond}.

In the following, we denote by $d(\cdot,\cdot)$ the Euclidean distance in 
$\C\cong \R^2$ between points or sets.

We state the analogous assumptions to those in Theorem~\ref{theoConv}.
Let $D,\widetilde{D}$ be two simply connected bounded 
domains in $\C$. Let $p_0\in D$ be some ``reference'' point.
For every $n\in\N$ 
assume that ${\mathscr C}_n$ and $\widetilde{\mathscr C}_n$ are two circle 
patterns satisfying conditions~(1)--(3) whose associated graphs $G_n$ and 
$\widetilde{G}_n$ are isomorphic. Moreover, we consider the associated kite 
patterns ${\mathscr K}_n$ and $\widetilde{\mathscr K}_n$ and assume that the 
sets ${\mathscr S}_n= \bigcup_{v\in V_n} S_n(v)= \bigcup_{e\in E_n} K_n(e)$ and 
$\widetilde{\mathscr S}_n=  \bigcup_{v\in V_n} S_n(v)= \bigcup_{e\in E_n} 
K_n(e)$
are simply connected and contained in $D$ and $\widetilde{D}$ respectively. 

Let $(\delta_n)_{n\in\N}$ be a sequence of positive numbers such that 
$\delta_n\searrow 
0$ for $n\to\infty$. For each $n\in\N$ assume that $r_n(v)<\delta_n/2$ for all 
$v\in V_n$ and that the Euclidean distance of ${\mathscr S}_n$ to the boundary 
$\partial D$ is smaller than $\delta_n$, 
i.e.\ $d({\mathscr S}_n,\partial D)<\delta_n$. We also suppose that 
$d(\widetilde{\mathscr S}_n,\partial \widetilde{D})<\delta_n$. 
Furthermore, let $p_0$ be covered by some $S_n(v)$ for every $n\in\N$. 
Let $f^\Diamond_n$ be a discrete conformal map as defined above
and assume that the closure of the image points 
$\overline{(f_n^\Diamond(p_0))}_{n\in\N}$ is compact in $\widetilde{D}$.

The following two lemmas and their proofs are modifications of the 
corresponding statements in~\cite[Section~2]{HeSch98}. 

\begin{lemma}\label{lemUnifProper}
 The maps $f_n^\Diamond$ are eventually uniformly proper onto $\widetilde{D}$, 
that is, given any compact set $\widetilde{K}\subset \widetilde{D}$ there is a 
compact set 
$K\subset D$ such that $(f^\Diamond_n)^{-1}(\widetilde{K})\subset K$ for 
sufficiently large $n$.
\end{lemma}
\begin{proof}
It is sufficient to consider a compact set $\widetilde{K}\subset \widetilde{D}$ 
which is connected and contains $\{f_n^\Diamond(p_0) : n\in\N\}$.

Let $\epsilon=d(\widetilde{K},\widetilde{D})$. Set $r_1=\min\{d(p_0,\partial 
D)\}/3$. Let $r_0\in(0,r_1/2)$ be a constant which is very 
small compared  to $r_1$ and $\epsilon$ and which will be specified later. 
Assume that $n$ is large enough such that $\delta_n<\min\{r_0,\eps\}/5$. Define 
$K=\{z\in D : d(z,\partial D)\geq r_0/2\}$. We will show that 
$(f^\Diamond_n)^{-1}(\widetilde{K})\subset K$. In particular, let $z_0\in 
{\mathscr S}_n\setminus K$, which implies $d(z_0,\partial D)<r_0/2$. We 
need to prove that $f_n^\Diamond(z_0)\not\in \widetilde{K}$.

For $\rho>0$ denote by $\ci(\rho)=\{z\in\C:d(z,z_0)=\rho\}$ the circle with 
radius $\rho$ about $z_0$.
Set $V_{\ci(\rho)}= \{v\in V_n: \text{interior}({S}_n(v)) \text{ 
intersects } \ci(\rho)\}$ and $L(\rho)=\sum_{v\in V_{\ci(\rho)}} 
\tilde{r}_{n}(v)$.

We now assume that there exists $\rho\in [r_0,r_1]$ such that 
$L(\rho)<\epsilon/4$ and prove this claim later.

First, by our assumptions we have $d(z_0,\partial D)\leq r_0/2 < r_1$ and
$d(\partial D, p_0)\geq 3r_1$. Thus $d(z_0, p_0)\geq 2 r_1$, so $\ci(\rho)$ 
separates $z_0$ from $p_0$.
Denote $A_n(\rho) =\bigcup_{v\in V_{\ci(\rho)}}{S}_n(v)$ and $\tilde{A}_n(\rho) 
=\bigcup_{v\in V_{\ci(\rho)}}\widetilde{S}_n(v)$.
By assumption $r_{n}(v)<\delta_n<r_0/2<r_1$, thus the set 
$A_n(\rho) \cup (\C\setminus {\mathscr S}_n)$ also separates 
$z_0$ from $p_0$ and is disjoint from $S_n(v)$ containing $p_0$.
Then $\tilde{A}_n(\rho) \cup (\C\setminus \widetilde{\mathscr S}_n)$ separates 
$f_n^\Diamond(z_0)$ from $f_n^\Diamond(p_0)$. Due to $d(z_0,\partial D)<\rho$ 
and diam$(\partial D)\geq 3r_1\geq 2\rho$ the circle $\ci(\rho)$ intersects the 
boundary $\partial D$. Therefore the sets ${A}_n(\rho)\cup (\C\setminus  
{\mathscr S}_n)$ and $\tilde{A}_n(\rho)\cup (\C\setminus 
\widetilde{\mathscr S}_n)$ 
are connected. By assumption each connected component of $\tilde{A}_n(\rho)$ 
has a diameter $\leq 2L(\rho)<\epsilon/2$ and contains boundary edges of 
$\widetilde{\mathscr S}_n$. We deduce that $\tilde{A}_n(\rho)$ is contained in 
a $(\delta_n+\epsilon/2)$-neighbourhood of the boundary 
$\partial\widetilde{D}$. As 
$\delta_n+\epsilon/2<\epsilon$ we obtain $\widetilde{K}\cap \tilde{A}_n(\rho)= 
\emptyset$.

Now consider $\widetilde{K}'$ the connected component of 
$\widetilde{K}\cap\widetilde{\mathscr S}_n$ which contains 
$f_n^\Diamond(p_0)$. We easily deduce that 
$f_n^\Diamond(z_0)\not\in \widetilde{K}'$. 
If $z_1\in\partial{\mathscr S}_n$ we also know that $d(z_1,D)\leq 
2\delta_n\leq r_0/2$. Using the same arguments as for $z_0$ we can deduce that $
f_n^\Diamond(z_1)\not\in \widetilde{K}'$. As $\widetilde{K}$ is connected we 
see that 
$\widetilde{K}'=\widetilde{K}$ and therefore $f_n^\Diamond(z_0)\not\in 
\widetilde{K}$.

Altogether, it only remains to show $\inf\{L(\rho) : \rho\in[r_0,r_1]\}< 
\epsilon/4$. To this end, denote ${\mathcal A}= \{z\in\C : r_0\leq d(z_0,z)\leq 
r_1\}$ and $V_{[r_0,r_1]}=\bigcup_{\rho\in[r_0,r_1]} V_{\ci(\rho)} =\{v\in V_n: 
\text{interior}({S}_n(v))\cap {\mathcal A}\not= \emptyset\}$. For every vertex 
$v\in V_{[r_0,r_1]}$ 
define the interval $[a_v,b_v]=\{\rho\in[r_0,r_1] : v\in A_n(\rho)\}$. 
Then we can estimate 
\begin{align*}
\inf\{L(\rho) : \rho\in[r_0,r_1]\}\log\frac{r_1}{r_0}
&\leq \int_{r_0}^{r_1} \frac{L(\rho)}{\rho}d\rho 
= \sum_{v\in V_{[r_0,r_1]}} \tilde{r}_n(v) \int_{a_v}^{b_v}\frac{1}{\rho} d\rho 
\\
&\leq \sum_{v\in V_{[r_0,r_1]}} \tilde{r}_n(v) \frac{b_v-a_v}{a_v} \\
&\leq \left(\sum_{v\in V_{[r_0,r_1]}} \tilde{r}_n(v)^2\right)^{1/2} 
\left(\sum_{v\in V_{[r_0,r_1]}} \frac{(b_v-a_v)^2}{a_v^2}\right)^{1/2}.
\end{align*}
For the last estimate we have used the Cauchy-Schwarz inequality. 
By Corollary~\ref{corEstDiscs} we know that $\sum_{v\in V_{[r_0,r_1]}} 
\tilde{r}_{{\mathscr C}_n}(v)^2  \leq 
\frac{C_0^2}{\pi}\textsc{area}(\widetilde{D})$. 
Furthermore, denote by $dA$ the Euclidean 
area element and by $Q_v$ a disc with largest radius which is contained in 
$S_n(v)\cap{\mathcal A}$.
Conditions~(2) and (3) imply that diam$(Q_v)\geq 
(1/\hat{C}_0) (b_v-a_v)$ for some constant $\hat{C}_0>0$.
Then as every point is covered at most $N=2$ times by interior points of $Q_v$ 
(according to condition~(1)), we deduce
\begin{align*}
 \int_{z\in{\mathcal A}} \frac{1}{|z-z_0|^2}dA
&\geq \frac{1}{2\hat{C}_0} \sum_{v\in V_{[r_0,r_1]}} \int_{Q_v} 
\frac{1}{|z-z_0|^2}dA \\
& \geq \frac{1}{2\hat{C}_0} \sum_{v\in V_{[r_0,r_1]}} 
\frac{1}{b_v^2}\textsc{Area}(Q_v)
\geq \frac{\pi}{8\hat{C}_0} \sum_{v\in V^*} \frac{(b_v-a_v)^2}{a_v^2}.
\end{align*}
For the last estimate we have used 
$b_v-a_v\leq\text{diam}(Q_v)\leq\delta_n$ and $a_v\geq r_0\geq \delta_n$ by our 
assumptions, so $b_v\leq 2 a_v$. Calculating the integral
\[\int_{z\in{\mathcal A}} \frac{1}{|z-z_0|^2}dA 
 =\int_{r_0}^{r_1} \int_{0}^{2\pi} \frac{1}{r^2}rdrd\varphi
=2\pi\log\frac{r_1}{r_0}, \]
we finally obtain
\begin{align*}
 \inf\{L(\rho) : \rho\in[r_0,r_1]\}
&\leq \left(\frac{C_0^2}{\pi}(\textsc{Area}(\widetilde{D}))\right)^{1/2} 
\left(\frac{16\hat{C}_0}{\log(r_1/r_0)}\right)^{1/2} <\frac{\epsilon}{4}
\end{align*}
if $r_0$ was chosen small enough. This completes the proof.
  \end{proof}

\begin{lemma}\label{lemdeltatild}
 Let $\tilde{\delta}_n$ denote the maximum diameter of the circles of 
$\widetilde{\mathscr C}_n$. Then $\tilde{\delta}_n\to 0$ for $n\to\infty$.
\end{lemma}
\begin{proof}
Let $\eps>0$. We will show that $\tilde{r}_n(v) <C_0\eps/2$ for all $v\in V_n$ 
if $n$ is large enough.

Let $\widetilde{K}\subset \widetilde{D}$ be a compact set such that 
$\{z\in\widetilde{D} : 
d(z,\partial\widetilde{D})\geq \eps/2\}\subseteq \widetilde{K}$. By 
Lemma~\ref{lemUnifProper}, there exists a compact set $K\subset D$ such that 
$(f^\Diamond_n)^{-1}(\widetilde{K})\subseteq K$ for sufficiently large $n$. 

Let $v\in V_n$ and $n$ be large enough. If ${S}_n(v)\cap K =\emptyset$ then
$\tilde{r}_n(v)\leq C_0\eps/2$ and also
$d(\tilde{c}_n(v),\partial\widetilde{D})\geq \tilde{r}_n(v)/C_0$. 

Set $\eps'=\min\{\eps,\text{diam}(\partial\widetilde{D})\}$ and
$r_1=d(K,\partial D)/3$. Let $r_0\in(0,r_1/3)$ be very small compared with 
$r_1$ and $\eps'$ and $n$ sufficiently large such that $\delta_n<r_0/3$.
Let $v\in V_n$ with ${B}_n(v)\cap K \not=\emptyset$. Then 
$d({B}_n(v),\partial D)\geq 3r_1-\delta_n > 2r_1$ and $d({\mathscr 
C}_n(v),\partial {\mathscr S}_n)> 2r_1-2\delta_n > r_1$.
Set $z_0=c_n(v)$. We proceed as in the proof of Lemma~\ref{lemUnifProper} and 
define $\ci(\rho)$, $V(\rho)$ and $\tilde{A}(\rho)$ for $\rho\in[r_0,r_1]$ in 
the 
same manner. Then we deduce that $\tilde{A}(\rho)$ separates 
$\widetilde{B}_n(v)$ from $\partial\widetilde{D}$.
Applying the same arguments as in the proof of Lemma~\ref{lemUnifProper} shows 
that for sufficiently small $r_0>0$ there exists $\rho\in[r_0,r_1]$ such that 
diam$(\tilde{A}(\rho))<\eps'/2$. As $\tilde{A}(\rho)$ separates 
$\widetilde{\mathscr C}_n(v)$ from $\partial\widetilde{D}$ and 
diam$(\partial\widetilde{D})\geq\eps'$ this implies that $\tilde{r}_n(v)\leq 
\eps'/4<\eps/2$.
  \end{proof}

The preceding lemmas imply the following corollary.

\begin{corollary}\label{corKtildeK}
Let $K\subset D$ and $\widetilde{K}\subset \widetilde{D}$ be two compact sets. 
Then 
the following statements hold for sufficiently large $n$.
\begin{enumerate}[(i)]
 \item $\text{area}({\mathscr S}_n)\supset K$.
\item $\text{area}(\widetilde{\mathscr S}_n)\supset \widetilde{K}$.
\item There exists a compact set $\widetilde{K}'\subset \widetilde{D}$ such 
that 
$f_n^\Diamond(K)\subset \widetilde{K}'$.
\end{enumerate}
\end{corollary}

\begin{remark}
 For unbounded simply connected domain we can consider the stereographic 
 projections of the pattern of circles to the sphere. Assume that these
 patterns of circles on $\Sp^2$ satisfy conditions~(1)--(3) analogously, where 
Euclidean distances have to be replaced by spherical distances.
 Then Lemmas~\ref{lemUnifProper} and~\ref{lemdeltatild} and 
 Corollary~\ref{corKtildeK} also hold for the corresponding discrete conformal 
maps
 similarly as in the proof by Schramm and He in~\cite{HeSch98}.
\end{remark}

\begin{remark}
 Lemmas~\ref{lemUnifProper} and~\ref{lemdeltatild} and 
Corollary~\ref{corKtildeK}
 also hold in the following case (with the same proofs).
 
 Consider two sequences of triangulations in $D$ and $\widetilde{D}$ with the
 same combinatorics whose
 angles are uniformly bounded from above and below. Let ${\mathscr C}_n$ and
 $\widetilde{\mathscr C}_n$ be the corresponding patterns of circles built from 
the 
 circumcircles and define the sets $B(v)$, $I_T(v)$ and $S_T(v)$ as
in Remark~\ref{remdefStriang}.
Set $f^\Diamond_n$ to be the piecewise linear map on the corresponding 
triangulations.
\end{remark}

\section{Convergence of circle patterns}\label{secConv}

In this section we prove Theorem~\ref{theoConv}. This convergence 
result holds for circle patterns with convex kites whose angles are all bounded 
uniformly away from zero and $\pi$. 
This is a natural generalization of circle patterns with regular combinatorics 
for example square grid or hexagonal, or of isoradial circle patterns with 
bounded intersection angles.
To be precise, we set

\begin{definition}\label{defqbound}
Let $q>1$. Let $\mathscr D$ be a b-quad-graph. A (planar) circle 
pattern ${\mathscr C}$ for $\mathscr D$ and some admissible labelling 
$\alpha:E(G)\to (0,\pi)$ is called {\em convex q-bounded circle pattern} if all 
kites of the corresponding kite pattern ${\mathscr K}$ are convex and if the 
ratios of the lengths of the diagonals of the kites are uniformly bounded in 
$[1/q, q]$.
\end{definition}

Of course, any finite circle pattern is q-bounded for some suitable q, so the 
above notion is more important for sequences of circle patterns or infinite 
circle patterns. Also note, 
that the uniform boundedness of the ratios of the lengths of the diagonals of 
the kites is equivalent to the uniform boundedness of the angles of the kites 
in $(c,\pi-c)$ for some $c>0$. Furthermore, q-boundedness implies uniform 
boundedness of the degree of the vertices in the corresponding graph $G$.

For convenience we recall the assumptions of Theorem~\ref{theoConv}:
Let $D$ and $\widetilde{D}$ be two simply connected bounded domains in $\C$. 
Let $p_0\in D$ be some ``reference'' point.
For $n\in\N$ let $({\mathscr D}_n)_{n\in\N}$ be a sequence of finite 
b-quad-graphs which are cell decompositions of $D$. 
Let  ${\mathscr C}_n$ and $\widetilde{\mathscr C}_n$ are two embedded 
convex q-bounded
(planar) circle patterns for ${\mathscr D}_n$ and some admissible labelling 
$\alpha_n$ whose kites all lie in $D$ and $\widetilde{D}$ respectively. 
Let $(\delta_n)_{n\in\N}$ be a sequence of positive numbers such that 
$\delta_n\searrow 
0$ for $n\to\infty$. For each $n\in\N$ assume that $r_n(v)<\delta_n/2$ for all 
$v\in V_n(G)$ .
Further suppose that $d({\mathscr S}_n,\partial D)<\delta_n$ and
$d(\widetilde{\mathscr S}_n,\partial \widetilde{D})<\delta_n$. 
Let $p_0$ be covered by a kite for every $n\in\N$ and let the 
closure of the image points $\overline{(f_n^\Diamond(p_0))}_{n\in\N}$ be 
compact in $\widetilde{D}$.

Our proof is inspired by ideas used in~\cite{RS87,HeSch96,Bue08}.
First note that for $n$ large enough the maps 
$f_n^\Diamond$ are eventually uniformly proper by Lemma~\ref{lemUnifProper}
and Lemma~\ref{lemdeltatild} implies that $\tilde{\delta}_n\to 0$. Also, 
Corollary~\ref{corKtildeK} holds.

Combining Corollary~\ref{corKtildeK} with the generalized maximum principle in 
the hyperbolic plane (Lemma~\ref{lemMaxHypGen}) we obtain bounds on the 
quotients of radii of both patterns.

\begin{lemma}\label{lemEstrn}
 Let $K\subset D$ be a compact set. Then 
there are constants $\hat{C}_1=\hat{C}_1(D,K)$ and 
$\hat{C}_2=\hat{C}_2(\widetilde{D},K)$, 
depending only on $D$, $\widetilde{D}$ and $K$, such that for $n$ large enough 
and 
all $v\in V_n$ with ${\mathscr C}_n(v)\subset K$ there holds
\begin{align}\label{eqestrn}
r_n(v)\geq \hat{C}_1 \tilde{r}_n(v)\qquad \text{and} \qquad 
 \tilde{r}_n(v)\geq \hat{C}_2 r_n(v).
\end{align}
\end{lemma}
\begin{proof}
 We start with the second estimate.

Let $\B(\rho)=\B_{z_0}(\rho)= 
\{z\in\C : |z-z_0|\leq \rho\}\subset \widetilde{D}$ be some 
closed disc contained in $ \widetilde{D}$ and denote 
$\B(\rho/2) =\{z\in\C : |z-z_0|\leq \rho/2\}$. Let $n$ be large 
enough such that $\B(\rho) \subset\widetilde{\mathscr S}_n$. Let 
$R>0$ be 
such that $D\subset\{z\in\C : |z|\leq R/2\}$. Consider the part ${\mathscr 
C}_n^{\B(\rho)}$ of the circle pattern whose image circles 
$\widetilde{\mathscr  C}_n(v)$ have non-empty intersection with 
the interior of $\B(\rho)$. Scale 
both patterns by $1/R$ and $1/\rho$ respectively. Now we can apply 
Lemma~\ref{lemMaxHypGen} and deduce that the hyperbolic radius of a circle in 
$\frac{1}{R}{\mathscr C}_n^{\B(\rho)}$ is smaller than the hyperbolic radius 
of the corresponding circle in $\frac{1}{\rho} \widetilde{\mathscr 
C}_n^{\B(\rho)}$. As hyperbolic and Euclidean radii are comparable for circles 
in $\frac{1}{2}\D$ we finally obtain $\tilde{r}_n(v)\geq 
\hat{C}_0\frac{\rho}{R}r_n(v)$ for some universal constant $\hat{C}_0$.

Now let $K\subset D$ be a compact set. By Corollary~\ref{corKtildeK} there 
exists a compact set $\widetilde{K}$ such that $f_n^\Diamond(K) 
\subset\widetilde{K}$ 
for $n$ large enough. As $\widetilde{K}$ is compact, it can be covered by 
finitely many closed discs $\B_{z_0}(\rho/2) =\{z\in\C : |z-z_0|\leq 
\rho/2\}$ such that $\B_{z_0}(\rho)= \{z\in\C : |z-z_0|\leq \rho\}$ is 
contained 
in 
$\widetilde{D}$. Applying the above reasoning for all these discs we deduce 
that the second inequality of~\eqref{eqestrn} holds for $n$ large enough. 

The first estimate can be obtained similarly by interchanging the roles of $D$ 
and $\widetilde{D}$.
  \end{proof}

\begin{corollary}
 Let $K\subset D$ be a compact set. Then the restricted homeomorphisms 
$\{f_n^\Diamond|_K\}$ form a $k$-quasiconformal normal family, where the 
constant $k$ only depends on $D$, $\widetilde{D}$ and $K$.
\end{corollary}
\begin{proof}
 The homeomorphisms $f_n^\Diamond$ are affine linear maps on every triangle 
obtained by dividing kite $K_e$ of ${\mathscr K}_n$ by the line corresponding 
to the edge $e=[v_0,v_1]\in E_n$. By construction the corresponding linear map 
has two eigenvalues $\tilde{r}_n(v_0)/r_n(v_0)$ and 
$\tilde{r}_n(v_1)/r_n(v_1)$. These two values are bounded from above and 
below by Lemma~\ref{lemEstrn} on compact sets. Also, the angle between the two
eigenvectors is bounded independently of $n$. Thus the quasiconformal 
distortion is bounded on compact sets if $n$ is large enough.
  \end{proof}

Using results from the theory of quasiconformal mappings (see~\cite{LV73} for 
example) we deduce that for 
every compact set $K\subset D$ there is a subsequence of 
$\{f_n^\Diamond|_K\}$ which converges uniformly on compact subsets of the 
interior int$(K)$ to some function $g_K$ which is $k$-quasiconformal or 
a mapping with exactly one or two values. Lemma~\ref{lemEstrn} implies that 
$g_K$ must be a $k$-quasiconformal homeomorphism. In fact, we can deduce 
from~\eqref{eqestrn} that the length of the image of a curve in 
$K$ is bounded from below (and above) by a constant $C>0$ times the original 
length:
\begin{align}\label{eqestgamma}
C\cdot \text{length}(\gamma)\leq 
\text{length}(f_n^\Diamond(\gamma))\leq\frac{1}{C} \cdot
\text{length}(\gamma).
\end{align}
In particular, the image of a disc of radius $R$ about $v_0$ contains at least 
a disc about $f_n^\Diamond(v_0)$ with radius $CR$.

For $v\in V_n$ define $u_n(v)=\tilde{r}_n(v)/r_n(v)$.
We have seen that the quasiconformal constant $k=k(n,K)$ for $f^\Diamond_n$ 
has an upper bound which depends on the maximum (and minimum) of the 
quotient ${u_n}(v_1)/u_n(v_0)$ for edges $e=[v_0,v_1]$ lying in $K$. We will 
show that $k(n,K)$ converges to $1$ for $n\to\infty$. This implies that $g_K$ 
is in fact conformal.

Define two Laplacians $\Delta$ and $\widetilde{\Delta}$ on $G_n$ 
by~\eqref{eqdeflap}, where 
\begin{align}
\mu([v_0,v_k])&= 2f_{\alpha_n({[v_0,v_k]})}'(\log ({r}_n(v_k)/{r}_n(v_0)))
\quad \text{and}\label{defmu1} \\
\tilde{\mu}([v_0,v_k])&= 2f_{\alpha_n({[v_0,v_k]})}'(\log 
(\tilde{r}_n(v_k)/\tilde{r}_n(v_0))). \label{defmu2}
\end{align} 

\begin{lemma}\label{lemsubharm}
For all interior vertices $v$ of $V_n$ there hold $\Delta u_n(v)\geq 0$ and 
$\widetilde{\Delta} (1/u_n)(v)\geq 0$, that is
 the functions $u_n$ and $1/u_n$ are subharmonic on $G_n$ (with respect to 
different Laplacians).
\end{lemma}
\begin{proof}
We only prove the first claim as the second case is analogous up to 
interchanging $r_n$ and $\tilde{r}_n$.

Recall that the radius functions $r_n$ and $\tilde{r}_n$ 
satisfy~\eqref{eqFgen} at all interior vertices. 
Considering a Taylor expansion in $u_n(v)$ at $u_n(v_0)$ we obtain
\begin{align*}
 0 &= \Biggl(\sum_{[v,v_0]\in E(G_n)} 
        f_{\alpha_n([v,v_0])}(\log \tilde{r}_n(v)- \log 
\tilde{r}_n(v_0))\Biggr) 
-\pi \\
&=\Biggl(\sum_{[v,v_0]\in E(G_n)} 
        f_{\alpha_n({[v,v_0]})}(\log {r}_n(v)- \log {r}_n(v_0) + \log 
u_n(v)-\log u_n(v_0))\Biggr) -\pi \\
&=\underbrace{\Biggl(\sum_{[v,v_0]\in E(G_n)} 
        f_{\alpha_n({[v,v_0]})}(\log {r}_n(v)- \log {r}_n(v_0))\Biggr) 
-\pi}_{=0 \text{ as }{\mathscr C}_n \text{ circle pattern}} \\
&\qquad +\sum_{[v,v_0]\in E(G_n)} 
        f_{\alpha_n({[v,v_0]})}'(\log {r}_n(v)- \log 
{r}_n(v_0))\frac{1}{u_n(v_0)}(u_n(v)-u_n(v_0)) \\
&\qquad +\frac{1}{2}\sum_{[v,v_0]\in E(G_n)} 
(f_{\alpha_n({[v,v_0]})}''(\xi_v)\frac{1}{t_v^2}- 
f_{\alpha_n({[v,v_0]})}'(\xi_v)\frac{1}{t_v^2}) (u_n(v)-u_n(v_0))^2.
\end{align*}
Here $\xi_v = \log {r}_n(v)- \log {r}_n(v_0) + \log t_v-\log u_n(v_0)$ and
$t_v=\lambda_v u_n(v_0) + (1-\lambda_v) u_n(v)$ for suitable 
$\lambda_v\in(0,1)$. Multiply the above equations by $u_n(v_0)$. Then the claim 
follows from the definition of the Laplacian in~\eqref{eqdeflap} if we can show 
that the last sum is non-positive. 

By Lemma~\ref{lemPropf} we have
\begin{align*}
 f_{\alpha_n({[v,v_0]})}''(\xi_v)- f_{\alpha_n({[v,v_0]})}'(\xi_v) &= 
 -\frac{\sin\alpha_n({[v,v_0]})(\sinh \xi_v +\cosh 
\xi_v-\cos\alpha_n({[v,v_0]}))}{2(\cosh\xi_v -\cos\alpha_n({[v,v_0]}))^2}
\end{align*}
In case $\alpha_n({[v,v_0]})\geq \pi/2$ we deduce that this term is zero or 
negative. In the case $0<\alpha_n({[v,v_0]})< \pi/2$ the convexity of all kites 
implies that $r_n(v)/r_n(v_0)\geq \cos\alpha_n({[v,v_0]})$ and
$\tilde{r}_n(v)/\tilde{r}_n(v_0)\geq \cos\alpha_n({[v,v_0]})$. Also note $\sinh 
\xi_v +\cosh \xi_v =\text{e}^{\xi_v}$. Therefore
\begin{align*}
 \text{e}^{\xi_v} -\cos\alpha_n({[v,v_0]})
&= \lambda_v \frac{r_n(v)}{r_n(v_0)} + 
(1-\lambda_v)\frac{\tilde{r}_n(v)}{\tilde{r}_n(v_0)} -\cos\alpha_n({[v,v_0]}) \\
&\geq \lambda_v\cos\alpha_n({[v,v_0]}) +(1-\lambda_v)\cos\alpha_n({[v,v_0]}) 
-\cos\alpha_n({[v,v_0]}) = 0.
\end{align*}
Thus the term $f_{\alpha_n({[v,v_0]})}''(\xi_v)- 
f_{\alpha_n({[v,v_0]})}'(\xi_v)$ is also non-positive in this case.
This finishes the proof.
  \end{proof}

\begin{remark}
The arguments in the preceding proof constitute an alternative 
proof of the maximum principle Lemma~\ref{lemMaxEuc} for the radius function 
for circle patterns with only convex kites.
\end{remark}

Consider the graph $G_n$ with with weights defined in~\eqref{defmu1} and 
\eqref{defmu2} respectively. We will show and use the fact that these two 
weighted networks are recurrent in the limit $n\to\infty$.

Consider discs $\B(v_0,R)= \B_{c_n(v_0)}(R)= \{z\in\C : |z-c_n(v_0)|\leq R\}
\subset D$ of fixed radius $R>0$ about $c_n(v_0)$ and $\tilde{\B}(v_0,R)=
\{z\in\C : |z-f^\Diamond_n(v_0)|\leq 
R\}\subset \widetilde{D}$ of fixed radius $R>0$ about $f^\Diamond_n(v_0)$. We 
assume that these discs are completely covered by kites in the respective 
patterns. 
Then taking the vertices whose corresponding centers of circles are contained 
in $B(v_0,R)$ we obtain a subgraph of $G_n$. Denote by
$G_n(v_0,R)$ its connected component containing $v_0$. Analogously, we 
define $\widetilde{G_n}(v_0,R)$. Furthermore, denote $\partial G_n(v_0,R)$ 
and $\partial \widetilde{G_n}(v_0,R)$ the boundary of $G_n(v_0,R)$ and 
$\widetilde{G_n}(v_0,R)$ in $G_n$ respectively.

\begin{lemma}\label{lemRec}
 $R_{\text{eff}}(v_0,\partial G_n(v_0,R))\to \infty$ and 
$R_{\text{eff}}(v_0,\partial \widetilde{G_n}(v_0,R))\to \infty$  for 
$n\to\infty$.
\end{lemma}
\begin{proof}
We will use~\eqref{eqdefReff} and ideas of the proof of Lemma~\ref{lemrec}. We 
will only prove the first case as both cases are analogous.

First we estimate the effective resistance between $\partial G_n(v_0,r)$ 
and $\partial G_n(v_0,r+\varepsilon)$, where $\varepsilon=2\delta_n$,
$r=2\delta_n k$ and 
$2\leq k\leq \lfloor \frac{R}{2\delta_n}\rfloor-3$. Here $\lfloor y\rfloor$ 
denotes the biggest integer smaller than $y$. 
Consider the function given by $f(v)=1$ on 
$V_n(v_0,r)$, by $f(v)=1-\frac{d(c_n(v),c_n(v_0))-r}{\varepsilon}$ 
on $A(r,\varepsilon):= V_n(v_0,r+\varepsilon)\setminus V_n(v_0,r)$ 
and $f(v)=0$ else. Note that $A(r,\varepsilon)\not=\emptyset$ and for all edges 
$[x,y]$ with length $\ell_n(x,y)=  
d(c_n(x),c_n(y))$ we have $|f(x)-f(y)|\leq \frac{1}{\varepsilon} \ell_n(x,y)$.
Denote by $E_A$ the edges with at least one vertex in 
$A(r,\varepsilon)$.
Then by~\eqref{eqdefReff} we obtain
\begin{align*}
\frac{1}{R_{\text{eff}}(G_n(v_0,r),\partial G_n(v_0,r+\varepsilon))}
&\leq \sum_{[x,y]\in E_A} \mu([x,y])(f(x)-f(y))^2 \\
&\leq \sum_{[x,y]\in E_A} \frac{1}{\varepsilon^2} \mu([x,y])\ell_n(x,y)^2.
\end{align*}
Recall from Remark~\ref{remArea} that $\mu([x,y])\ell_n(x,y)^2$ is the area of 
the kite in ${\mathscr K}_n$ corresponding to the edge $[x,y]$. 
Furthermore $\ell_n(x,y)\leq 2\delta_n$. Thus we may estimate the sum by
$\frac{\pi}{\varepsilon^2}((r+\varepsilon+2\delta_n)^2- (r-2\delta_n)^2) 
=\pi (3+6k)$
using the area of the annulus $\{z\in\C: 
r-2\delta_n<|z-c_n(v_0)|<r+\varepsilon+2\delta_n\}$.
Summing up, we get
\[ R_{\text{eff}}(v_0,\partial V_n(v_0,R))\geq \sum_{l=1}^{\lfloor 
\frac{R}{4\delta_n}\rfloor-2}\frac{1}{2(3+6l)\pi}.\]
As $\delta_n\to 0$ for $n\to\infty$ the claim follows.
  \end{proof}

\begin{lemma}
Fix $q_0\in D$ and let $v_0\in V_n$ be a vertex nearest to $q_0$. Then for any 
vertex $v_1$ adjacent to $v_0$ in $G_n$ we have $u_n(v_1)/u_n(v_0)\to 1$ for 
$n\to\infty$. Furthermore, this convergence is uniform on closed discs 
contained in $D$.
\end{lemma}
\begin{proof}
Let $\B(v_0,R)\subset D$ be fixed. Let $h_n$ be the unique function which is 
harmonic with respect to $\Delta$ on ${G_n}(v_0,R)\setminus\{v_0\}$ with 
values $h_n(v_0)=1$ and $h_n(v)=0$ outside 
${G_n}(v_0,R)$. This function minimizes the energy 
in~\eqref{eqdefReff} for $Z=v_0$ and $A=V_n\setminus V_n(v_0,R)$, 
see~\cite{LP}. Let $v_1$ be adjacent to $v_0$. The preceding lemma implies 
that $h_n(v_1)\to h_n(v_0)=1$ for $n\to\infty$ as the weights $\mu([x,y])$ are 
uniformly bounded from above and below. 
An analogous reasoning applies for $\tilde{h}_n$ being the unique function 
which is harmonic with respect to $\widetilde{\Delta}$ on 
$\widetilde{G_n}(v_0,\tilde{R})\setminus\{v_0\}$ with values 
$\tilde{h}_n(v_0)=1$ and $\tilde{h}_n(v)=0$ outside 
$\widetilde{G_n}(v_0,\tilde{R})$. We choose $\tilde{R}=RC$ for 
the constant $C$ of estimate~\eqref{eqestgamma}.

Lemma~\ref{lemEstrn} implies that $u_n$ and $1/u_n$ are bounded on 
${G_n}(v_0,R)$ with bounds independent of $n$, that is 
$|u_n|\leq M$ and $|1/u_n|\leq \widetilde{M}$. Consider 
\[v_n=u_n-(M(1-h_n)+u_n(v_0)h_n)\quad \text{and}\quad 
\tilde{v}_n=1/u_n-(\widetilde{M}(1-\tilde{h}_n)+\tilde{h}_n/u_n(v_0)).\] 
By construction we have $v_n(v_0)=0=\tilde{v}_n(v_0)$. Also 
$v_n|_{\partial_n}\leq 0$ and $\tilde{v}_n|_{\tilde{\partial}_n}\leq 0$ 
restricted to 
the boundary $\partial_n= $ boundary vertices of ${G_n}(v_0,R)$ and 
$\tilde{\partial}_n= $ boundary vertices of $\widetilde{G_n}(v_0,\tilde{R})$ . 
Furthermore $v_n$ is subharmonic with respect to $\Delta$ on 
${G_n}(v_0,R)\setminus\{v_0\}$ and $\tilde{v}_n$ is subharmonic with respect 
to $\widetilde{\Delta}$ on $\widetilde{G_n}(v_0,\tilde{R})\setminus\{v_0\}$. By 
the maximum principle for subharmonic functions on graphs
(see~\cite{Woe00} or~\cite{LP}) we deduce that for neighboring vertices $v_1$ 
of $v_0$ we have
\[ u_n(v_1)\leq u_n(v_0)h_n(v_1) +M(1-h_n(v_1))\quad\text{and}\quad
 \frac{1}{u_n(v_1)}\leq \frac{\tilde{h}_n(v_1)}{u_n(v_0)} 
+\widetilde{M}(1-\tilde{h}_n(v_1)). \]
This implies $u_n(v_1)/u_n(v_0)\to 1$ for $n\to\infty$. Also, the proof of 
Lemma~\ref{lemRec} and the previous reasoning show that in fact this 
convergence is uniform on discs about $v_0$.
  \end{proof}

\begin{corollary}
 The $k$-quasiconformal homeomorphisms $g_K$ are in fact conformal.
\end{corollary}

\begin{proof}[Proof of Theorem~\ref{theoConv}.]
We have already proven that for every fixed compact set $K\subset D$ there is a 
subsequence of $\{f_n^\Diamond|_K\}$ which converges uniformly on compact 
subsets of the interior int$(K)$ to some function $g_K$ which is conformal on 
int$(K)$.
Consider an exhaustion of $D$ by compact set $K_j$ such that 
$K_j\subset \text{int}(K_{j+1})$ and $\bigcup_{j}K_j =D$. Using a diagonal 
process we can obtain a subsequence which converges uniformly on all compact 
sets in $D$ to a conformal map $g$ on $D$. For simpler notation we denote this 
subsequence again by $\{f_n^\Diamond\}$.

It remains to be shown that $g(D)=\widetilde{D}$. Indeed let 
$w\in\widetilde{D}$. 
As $d(\widetilde{\mathscr S}_n,\partial\widetilde{D})\leq\delta_n 
\to 0$ we 
deduce from Corollary~\ref{corKtildeK} that $(f_n^\Diamond)^{-1}(w)\in 
\text{int}(K_j)$ for some $j$ and $n$ large enough. 
Estimate~\eqref{eqestgamma} implies that $f_n^\Diamond$ are equicontinuous for 
$n$ large enough. Thus this family of $k$-quasiconformal mappings on $K_j$ is 
also Hoelder continuous with Hoelder exponent $\alpha\in(0,1)$.
Now let $z_n=(f_n^\Diamond)^{-1}(w)\in K_j$. As $K_j$ is compact a subsequence 
$\{z_{n_l}\}$ of $\{z_n\}$ will converge to some $z\in K_j$. Thus we have
\[|f_{n_l}^\Diamond(z)-w|= |f_{n_l}^\Diamond(z)- f_{n_l}^\Diamond(z_{n_l})|
 \leq A|z-z_{n_l}|^\alpha
\]
with some constant $A$ depending only on $k$ and $K_j$.
This implies $f_{n_l}^\Diamond(z)\to w$ and therefore $f_n^\Diamond(z)\to w$ as 
$\{f_n^\Diamond\}$ converges by our previous reasoning. Thus $w\in g(D)$ and 
hence $g$ is surjective as $w$ can be chosen arbitrarily.

This completes the proof of Theorem~\ref{theoConv}.
  \end{proof}

\section*{Acknowledgement}
The author is grateful to Boris Springborn for useful discussions and
advice.

This research was supported by the DFG Collaborative Research Center
TRR~109 ``Discretization in Geometry and Dynamics''.



\bibliographystyle{amsalpha} 
\bibliography{RigidityConvergence}

\normalsize

\appendix
\section{Proof of Lemma~\ref{lemCompRad}}\label{secProofLemma}
\normalsize
In this appendix we prove Lemma~\ref{lemCompRad} by suitably adapting the 
proof of Lemma~6.1 in~\cite{He99} using estimates of Section~\ref{secVELest}.

Let $\mathscr C$ be an infinite circle pattern. Denote
\[\tau(v)=\frac{r(v)}{d(0,B(v))}\in (0,\infty]\quad \text{for } v\in V\qquad
\text{and}\qquad
\tau({\mathscr C})= \limsup_{v_k\to\infty} \tau(v_k)\in[0,\infty]\]
where we have arranged the vertices 
of $V$ into a sequence $(v_k)$. Note that if $\mathscr C$ is locally finite and 
$f:\C\to\C$ is a similarity then $\tau(f({\mathscr C}))= \tau({\mathscr C})$.

Let $v_0\in V$. Without loss 
of generality we may assume that the centers $c(v_0)=0$ and $\tilde{c}(v_0)=0$ 
are 
placed at the origin. Furthermore we may assume that $r(v_0)=\tilde{r}(v_0)$ by 
suitable scaling. Let $v_1\in V$ be another vertex.

\paragraph{Case (i):}  Assume that $\tau({\mathscr C})<\infty$. Then there 
exists a constant $\delta>3$ and a finite subset $V_0\subset V$, containing 
$v_0$ and $v_1$ such that for all $v\in V\setminus V_0$ there holds
\[\frac{r(v)}{d(0,B(v))} =\tau(v)\leq \frac{\delta}{3}.\]
Let $R_0>0$ be large enough such that $B(v)\subset \B(R_0)$ for all $v\in V_0$. 
Then for all $R>R_0$ we have $V_{\ci(R)}\cap V_{\ci(\delta R)} =\emptyset$ as 
$\tau(v)\leq \delta/3$.

Set $R_j=\delta^jR_0$ and $V_j=V_{\ci(R_j)}$ for $j\geq 1$. Then for all $0\leq 
i_1<i_2<i_3$ the set $V_{i_2}$ separates $ V_{i_1}$ and $V_{i_3}$ and 
$\vel(V_2,V_{2k-1})\geq \sum_{j=1}^{k-1} \vel(V_{2j},V_{2j+1})$ by 
Lemma~\ref{lemEstSumVel}. Furthermore, Lemma~\ref{lemEstVel} implies that 
$\vel(V_{2j},V_{2j+1})\geq C_2$ and thus $\vel(V_2,V_{2k-1})\geq (k-1) C_2$.
Choose $k=2+\lceil C_6/C_2 \rceil$, where $\lceil x \rceil$ is the smallest 
integer $\geq x$. Then $\vel(V_2,V_{2k-1})> C_6$. Now we apply 
Lemma~\ref{lemIntRad} to the circle pattern $\widetilde{\mathscr C}$. So there 
is $\tilde{R}>0$ such that for all $\tilde{\rho}\in[\tilde{R},2\tilde{R}]$ the 
set $\tilde{V}_{\ci(\tilde{\rho})}= \{v\in V: \tilde{S}(v)\cap\ci(\tilde{\rho}) 
\not=\emptyset\}$ separates $V_2$ and $V_{2k-1}$. Then 
$\tilde{V}_{\ci(\tilde{\rho})}$ also separates $V_1$, and $V_{2k}$ and 
$\tilde{V}_{\ci(\tilde{\rho})}\cap V_1 \subset V_2\cap V_1 =\emptyset$ as well 
as 
$\tilde{V}_{\ci(\tilde{\rho})}\cap V_{2k} \subset V_{2k-1}\cap V_{2k} 
=\emptyset$. 
This implies that $\tilde{B}(v)\subset \B_{int}(\tilde{\rho})$ for all $v\in 
V_1$ 
and $\tilde{B}(v)\subset \C\setminus\B_{int}(2\tilde{\rho})$ or all $v\in 
V_{2k}$.

Consider the subgraph $G_1$ of $G$ consisting of all vertices $v$ such that 
$B(v)\cap \B(R_1)\not =\emptyset$. Then 
$\tilde{B}(v)\subset\B(\tilde{R})\subset 
\B(2\tilde{R})$. Denote by $r_{hyp}(B)\subset\D$ the hyperbolic radius of 
the 
disc $B$. We deduce from Lemma~\ref{lemMaxHypGen} that
\[r_{hyp}(\textstyle\frac{1}{2\tilde{R}}\tilde{B}(v)) \leq 
r_{hyp}(\frac{1}{{R_1}}{B}(v))\]
holds for all vertices $v$ of $G_1$.

Similarly, consider the subgraph $G_2$ of $G$ consisting of all vertices $v$ 
such that $\tilde{B}(v)\cap \B(2\tilde{R})\not =\emptyset$.  Then 
${B}(v)$ is contained in the interior of $\B({R}_{2k})$ and for all vertices 
$v$ 
of $G_2$ we deduce from Lemma~\ref{lemMaxHypGen} that
\[r_{hyp}(\textstyle \frac{1}{2\tilde{R}}\tilde{B}(v)) \geq 
r_{hyp}(\frac{1}{{R_{2k}}}{B}(v)).\]

As the discs $\frac{1}{2\tilde{R}}\tilde{B}(v),\frac{1}{{R_1}}{B}(v), 
\frac{1}{2\tilde{R}}\tilde{B}(v),\frac{1}{{R_{2k}}}{B}(v)$ are all contained in 
$\frac{1}{2}\D$ the hyperbolic and Euclidean radii are comparable. In 
particular, there exists an absolute constant $\hat{C}_{0}>0$ such that 
\[\frac{1}{2\tilde{R}}\tilde{r}(v_j)\leq 
\hat{C}_{0} \frac{1}{{R_1}}{r}(v_j)\quad \text{and}\quad 
\hat{C}_{0} \frac{1}{2\tilde{R}}\tilde{r}(v_j) \geq \frac{1}{{R_{2k}}}{r}(v_j)= 
\frac{\delta^{-2k+1}}{R_1} r(v_j)\] 
hold for $j=0,1$. As $r(v_0)=\tilde{r}(v_0)$ 
we see that $\frac{1}{2\tilde{R}}\leq \hat{C}_{0} \frac{1}{{R_1}}$ and 
$\hat{C}_{0} \frac{1}{2\tilde{R}}\geq \frac{1}{{R_{2k}}}$, therefore
\[C=\hat{C}_{0}^2\delta^{2k-1} \geq \hat{C}_{0}\frac{2\tilde{R}}{{R_1}} \geq 
\frac{\tilde{r}(v_1)}{r(v_1)} \geq \delta^{-2k+1} 
\frac{2\tilde{R}}{\hat{C}_{0}{R_1}} 
\geq \frac{\delta^{-2k+1}}{\hat{C}_{0}^2}=\frac{1}{C}.\]

\paragraph{Case (ii):} Now assume that $\tau({\mathscr C})=\infty$. We consider 
the infinite set $W=\{v\in V: \tau(v)\geq 1\}$. Define $\tilde{\tau}$ 
analogously 
as $\tau$. We start with the following claim:
\begin{equation}\label{eqhattau}
 \limsup_{v\in W, v\to \infty} \tilde{\tau}(v)>0.
\end{equation}
\begin{proof}[Proof of~\eqref{eqhattau}]
 Suppose the contrary, that is $\limsup_{v\in W, v\to \infty} 
\tilde{\tau}(v)=0$. 
Then for all $0<\eps<1/2$ there is a finite subset $V_0\ni v_0$, containing 
also all neighbors of $v_0$, such that for 
all vertices $v\in W\setminus V_0=:W'$ there holds
\begin{equation}\label{eqcontass}
\frac{\tilde{r}(v)}{d(0,\tilde{B}(v))} =\tilde{\tau}(v)<\eps.
\end{equation}
Denote by $\widehat{G}$ the subgraph of $G$ which contains no vertices of 
$W'$ 
(or edges ending in $W'$). Let $R_0>0$ be large enough such that $B(v)\subset 
\B(R_0)$ for all $v\in V_0$. Set $R_j=4^j R_0$ and $V_j=V_{\ci(R_j)}$ for  
$j\geq 1$. Then $V_i\cap V_j\subset W'$ for all $i\not=j$ and for all $0\leq 
i_1<i_2<i_3$ the set $V_{i_2}$ separates $ V_{i_1}$ and $V_{i_3}$.

Denote by $\vel_{\widehat{G}}(V_i,V_j)$ the vertex extremal length between 
$V_i\cap (V\setminus W')$ and $V_j\cap (V\setminus W')$ in $\widehat{G}$. Set 
$k=2+\lceil 2C_6/C_2\rceil$. Then we obtain as in case~(i)
\[\vel_{\hat{G}}(V_2,V_{2k-1}) \geq \sum_{i=1}^{k-1} 
\vel_{\hat{G}}(V_{2i},V_{2i+1}) \geq (k-1) C_2 > 2C_6.\]

Assume that there exists $\tilde{R}>0$ such that for all 
$\rho\in[\tilde{R},2\tilde{R}]$ the set $\tilde{V}_{\ci(\rho)}$ separates $V_2$ 
from $V_{2k-1}$ in $G$. Then $V_2\cap V_{2k-1} \subset 
\tilde{V}_{\ci(\tilde{R})}\cap 
\tilde{V}_{\ci(2\tilde{R})}\cap W'=\emptyset$ due to~\eqref{eqcontass}. As 
$R_0$ is arbitrary  this implies that $\tau({\mathscr C})\leq 
\frac{4^{2k-3}-1}{2}<\infty$ contradicting our assumption.
It now remains to prove the existence of $\tilde{R}>0$.
\hfill  
\begin{proof}[Existence of $\tilde{R}>0$]
Let $U_{2,2k-1}$ be the subset of vertices which is separated from $V_0$ by 
$V_2$ and from $\infty$ by $V_{2k-1}$, that is $U_{2,2k-1}=\{v\in V: S(v)\cap 
\{z\in\C: R_2<|z|<R_{2k-1}\}\not= \emptyset\}$. Denote $W''=U_{2,2k-1}\cap W$. 
As $k$ is fixed, $\tau(v)\geq 1$ for $v\in W$ and by condition~(2) the 
number of vertices in $W''$ is bounded by a universal constant, say $|W''|\leq 
M$. 

Define $\tilde{R}=\min\{ R: \tilde{B}(v)\cap \B(R)\not=\emptyset \text{ for all 
} v\in V_2\} >0$. Without loss of generality we may assume that $\tilde{R}=1$.
Let $\gamma^*\in\Gamma_G^*(V_2,V_{2k-1})$. We deduce as in the proof of
Lemma~\ref{lemIntRad} that the diameter of $\tilde{\Sp}(\gamma^*)$ is $\geq 
\tilde{R}=1$.

Assume that there exists $\rho_1\in[1,2]$ such that $\tilde{V}_{\ci(\rho_1)}$ 
does not separate $V_2$ and $V_{2k-1}$. Thus there exists a path $\gamma_0$ 
from $V_2$ to $V_{2k-1}$ with $\gamma_0\cap \tilde{V}_{\ci(\rho_1)}=\emptyset$. 
Again as in the proof of Lemma~\ref{lemIntRad} this implies that 
$\tilde{S}(v_0)$ 
is contained in the interior of $\B(\rho_1)$, so $\tilde{\Sp}(\gamma^*)\cap 
\B(\rho_1)\supset \tilde{\Sp}(\gamma^*\cap \gamma_0)\not= \emptyset$.

Define $\eta(v)= \min\{2\tilde{r}(v),6\}$ 
if the area of the lunar region $\textsc{Area}(\tilde{B}(v)\cap 
\B(3))\geq\frac{1}{5}\pi \tilde{r}(v)^2$ analogously as in the proof of 
Lemma~\ref{lemIntRad}. If $v\in V_{\ci(3)}$ and  if 
$\textsc{Area}(\tilde{B}(v)\cap \B(3))<\frac{1}{5}\pi \tilde{r}(v)^2$ 
let $\eta(v)$ be the maximum of the length of the arc $\ci(3)\cap \tilde{S}(v)$ 
and of the maximum of the  distance of a point 
in $\tilde{S}(v)\cap \B(3)$ to the arc $\ci(3)\cap \tilde{S}(v)$ as in 
condition~(3). Else set $\eta(v)=0$. Then for any connected curve $\gamma^*\in 
\Gamma^*_G(V_2,V_{2k-1})$ we have $\sum_{v\in\gamma^*}\eta(v) \geq 1$ as in the 
proof of Lemma~\ref{lemIntRad}. By our assumptions on $V_2,V_{2k-1}$ and $G$ 
every curve $\gamma^*\in \Gamma^*_G(V_2,V_{2k-1})$ contains a connected 
subcurve, so the estimate holds for all such curves $\gamma^*$.

By our assumption we know that $\frac{\tilde{r}(v)}{d(0,\tilde{B}(v))}<\eps$ 
for $v\in W'\supset W''$, so $\eta(v)< 6\eps$. This implies $\sum_{v\in W''} 
\eta(v)\leq 6\eps M$ as $|W''|=M$ is constant.

Let $\beta^*\in \Gamma^*_{\tilde{G}}(V_2\cap V\setminus W', V_{2k-1}\cap 
V\setminus W')$. Then $\gamma^*=\beta^*\cup W''\in \Gamma^*_G(V_2,V_{2k-1})$ 
and $\sum_{v\in\beta^*}\eta(v) \geq 1-6\eps M$. Choose $\eps>0$ small enough 
such that $1-6\eps M\geq 1/\sqrt{2}$. Then $(\sqrt{2}\eta)$ is 
$\Gamma^*_{\tilde{G}}(V_2\cap V\setminus W', V_{2k-1}\cap V\setminus 
W')$-admissible and we obtain
\[\vel_{\tilde{G}}(V_2,V_{2k-1})=\m(\Gamma^*_{\tilde{G}}(V_2\cap V\setminus W', 
V_{2k-1}\cap V\setminus W'))\leq \text{area}(\sqrt{2}\eta)=2\text{area}(\eta).\]
Now $2\text{area}(\eta)\leq 2\cdot \frac{9}{4C_2}=2C_6$ can be bounded from 
above similarly as in the proof of Lemma~\ref{lemIntRad}. This contradicts our 
choice of $k$ and finishes the proof on existence of $\tilde{R}$.
  \end{proof}

Thus we have shown~\eqref{eqhattau}, that is there exists 
$\delta\in(0,\frac{1}{3})$ such that
\[\limsup_{v\in W, v\to \infty} \tilde{\tau}(v)>3\delta>0.\]
Let $u_k$ be a sequence of pairwise disjoint vertices satisfying $\tau(u_k)\geq 
1$ and $\tilde{\tau}(u_k)\geq 3\delta$. As $\mathscr C$ and 
$\widetilde{\mathscr 
C}$ are both locally finite in $\C$ we deduce that $\text{dist}(0,B(u_k))\to 
\infty$ and $\text{dist}(0,\tilde{B}(u_k))\to \infty$.

Let $\rho_0>0$ be large enough such that the discs $B(v_0), B(v_1), 
\tilde{B}(v_0), \tilde{B}(v_1)$ are all contained in $\B(\rho_0)$. Choose 
$u_k=:u$ 
such that
\[d=d(0,B(u))\geq 100\rho_0/\delta^2\quad \text{and}\quad
 \tilde{d}=d(0,\tilde{B}(u))\geq 100\rho_0/\delta^2. \]
See Figure~\ref{figMoebius} for an illustration.
\begin{figure}[tb]
\begin{center}
\resizebox{0.6\textwidth}{!}{\input{BrhoBv.pspdftex}}
\end{center}
 \caption{Illustration of the configuration of $\B(\rho_0)$ and 
$B(u)$.}\label{figMoebius}
\end{figure}
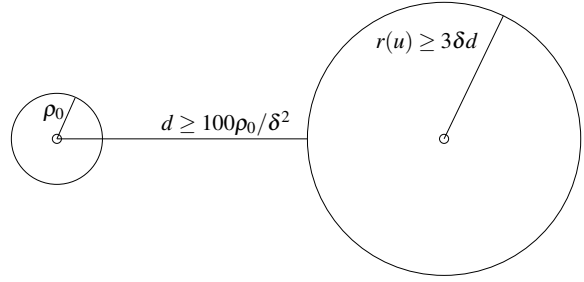
Let $F$ be a M\"obius transformation satisfying $F(0)=0$ and 
$F(\tilde{\C}\setminus B(u))=\D$ and let $\tilde{F}$ be a M\"obius 
transformation 
satisfying $\tilde{F}(0)=0$ and $\tilde{F}(\hat{\C}\setminus \tilde{B}(u))=\D$.
Then the hyperbolic distance between $0$ and $F(\infty)$ is bigger than the 
hyperbolic distance between $0$ and $r(u)/(d+\rho_0+r(u))$. Therefore 
\begin{align*}
|F(\infty)|&= \frac{r(u)}{d+r(u)} =\frac{1}{\frac{1}{\tau(u)} +1} \geq 
\frac{1}{2} > \delta\qquad \text{and} \\
|\tilde{F}(\infty)|&= \frac{1}{\frac{1}{\tilde{\tau}(u)} +1} \geq 
\frac{1}{\frac{1}{3\delta}+1} > \delta.
\end{align*} 
by analogous reasoning. As $\rho_0<\delta^2 d/100$ and 
$\tilde{\rho}_0<\delta^2 \tilde{d}/100$ the two discs $F(\B(\rho))$ and 
$\tilde{F}(\tilde{\B}(\rho))$ lie in $\frac{\delta}{2}\D$. This is even more 
true for the discs $F(B(v_0))$, $F(B(v_1))$, $\tilde{F}(\tilde{B}(v_0))$, 
$\tilde{F}(\tilde{B}(v_1))$. From our assumption $d=d(0,B(u))\geq 
100\rho_0/\delta^2 \geq 
100\rho_0$ we deduce that $|F'(z_1)/F'(z_2)|\leq 2$ holds for all $z_1,z_2\in 
\B(\rho_0)$. An analogous statement holds for $\tilde{F}$. Therefore, the 
ratios of radii are bounded:
\[ \frac{\text{radius}(F(B(v_1)))}{\text{radius}(F(B(v_0)))} \left/ 
\frac{r(v_1)}{r(v_0)} \right.\in[\frac{1}{4},4] \quad \text{and}\quad
\frac{\text{radius}(\tilde{F}(\tilde{B}(v_1)))}{\text{radius}(\tilde{F} 
(\tilde{B}(v_0)))} \left/ 
\frac{\tilde{r}(v_1)}{\tilde{r}(v_0)} \right.\in[\frac{1}{4},4].\]
As $r(v_0)=\tilde{r}(v_0)$ it only remains to show that 
$\frac{\text{radius}(F(B(v_1)))}{\text{radius}(F(B(v_0)))}$ and
$\frac{\text{radius}(\tilde{F}(\tilde{B}(v_1)))}{\text{radius}(\tilde{F} 
(\tilde{B}(v_0)))}$ are comparable. 

First compare $\frac{1}{\delta} F(\mathscr 
C)$ with $\tilde{F}(\widetilde{\mathscr C})$. As $F(\infty)>\delta$ and 
$F(\mathscr 
C)$ is locally finite in $\C\setminus\{F(\infty)\}$ there is only a finite 
number of circles in $\frac{1}{\delta}F(\mathscr C)$ which intersect $\D$. Thus 
Lemma~\ref{lemMaxHypGen} implies $r_{hyp}(\frac{1}{\delta} F(B(v_j))) \geq 
r_{hyp}(\tilde{F}(\tilde{B}(v_j)))$for $j=0,1$.
Analogously, we see $r_{hyp}(F(B(v_j))) \leq r_{hyp}(\frac{1}{\delta} 
\tilde{F}(\tilde{B}(v_j)))$ for $j=0,1$.
As in the proof of the first case~(i) the claim now follows.
  \end{proof}

\end{document}

%% file: DefPhi.pspdftex
\begin{picture}(0,0)%
\includegraphics{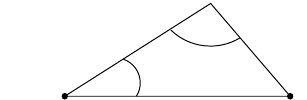}%
\end{picture}%
\setlength{\unitlength}{4144sp}%
\begingroup\makeatletter\ifx\SetFigFont\undefined%
\gdef\SetFigFont#1#2#3#4#5{%
  \reset@font\fontsize{#1}{#2pt}%
  \fontfamily{#3}\fontseries{#4}\fontshape{#5}%
  \selectfont}%
\fi\endgroup%
\begin{picture}(2325,748)(1111,-3497)
\put(3421,-3436){\makebox(0,0)[lb]{\smash{{\SetFigFont{8}{9.6}{\familydefault}{\mddefault}{\updefault}{\color[rgb]{0,0,0}$c_{\text{hyp}}(v_1)$}%
}}}}
\put(3080,-3084){\makebox(0,0)[lb]{\smash{{\SetFigFont{8}{9.6}{\familydefault}{\mddefault}{\updefault}{\color[rgb]{0,0,0}$r_{\text{hyp}}(v_1)$}%
}}}}
\put(2635,-2963){\makebox(0,0)[lb]{\smash{{\SetFigFont{8}{9.6}{\familydefault}{\mddefault}{\updefault}{\color[rgb]{0,0,0}$\theta$}%
}}}}
\put(1666,-3076){\makebox(0,0)[lb]{\smash{{\SetFigFont{8}{9.6}{\familydefault}{\mddefault}{\updefault}{\color[rgb]{0,0,0}$r_{\text{hyp}}(v_0)$}%
}}}}
\put(1126,-3436){\makebox(0,0)[lb]{\smash{{\SetFigFont{8}{9.6}{\familydefault}{\mddefault}{\updefault}{\color[rgb]{0,0,0}$c_{\text{hyp}}(v_0)$}%
}}}}
\put(2008,-3367){\makebox(0,0)[lb]{\smash{{\SetFigFont{8}{9.6}{\familydefault}{\mddefault}{\updefault}{\color[rgb]{0,0,0}$\varphi_{e}$}%
}}}}
\end{picture}%

%% file: GenPhi.pspdftex
\begin{picture}(0,0)%
\includegraphics{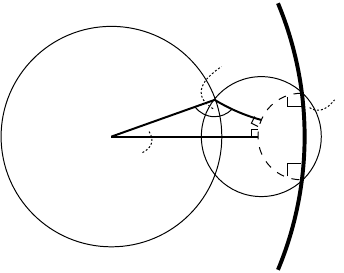}%
\end{picture}%
\setlength{\unitlength}{4144sp}%
\begingroup\makeatletter\ifx\SetFigFont\undefined%
\gdef\SetFigFont#1#2#3#4#5{%
  \reset@font\fontsize{#1}{#2pt}%
  \fontfamily{#3}\fontseries{#4}\fontshape{#5}%
  \selectfont}%
\fi\endgroup%
\begin{picture}(2616,2077)(5010,-3445)
\put(6849,-2178){\makebox(0,0)[lb]{\smash{{\SetFigFont{7}{8.4}{\familydefault}{\mddefault}{\updefault}{\color[rgb]{0,0,0}$\delta$}%
}}}}
\put(7611,-2153){\makebox(0,0)[lb]{\smash{{\SetFigFont{8}{9.6}{\familydefault}{\mddefault}{\updefault}{\color[rgb]{0,0,0}$\beta$}%
}}}}
\put(6671,-1823){\makebox(0,0)[lb]{\smash{{\SetFigFont{8}{9.6}{\familydefault}{\mddefault}{\updefault}{\color[rgb]{0,0,0}$\theta$}%
}}}}
\put(6031,-2176){\makebox(0,0)[lb]{\smash{{\SetFigFont{8}{9.6}{\familydefault}{\mddefault}{\updefault}{\color[rgb]{0,0,0}$r_{\text{hyp}}$}%
}}}}
\put(7246,-1501){\makebox(0,0)[lb]{\smash{{\SetFigFont{8}{9.6}{\familydefault}{\mddefault}{\updefault}{\color[rgb]{0,0,0}$\partial\D$}%
}}}}
\put(5910,-2635){\makebox(0,0)[lb]{\smash{{\SetFigFont{8}{9.6}{\familydefault}{\mddefault}{\updefault}{\color[rgb]{0,0,0}$\varphi_{e}^g$}%
}}}}
\end{picture}%

%% file: IntersectionCircle.pspdftex
\begin{picture}(0,0)%
\includegraphics{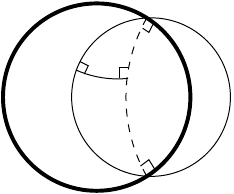}%
\end{picture}%
\setlength{\unitlength}{4144sp}%
\begingroup\makeatletter\ifx\SetFigFont\undefined%
\gdef\SetFigFont#1#2#3#4#5{%
  \reset@font\fontsize{#1}{#2pt}%
  \fontfamily{#3}\fontseries{#4}\fontshape{#5}%
  \selectfont}%
\fi\endgroup%
\begin{picture}(1764,1472)(2425,-3058)
\put(3782,-1861){\makebox(0,0)[lb]{\smash{{\SetFigFont{6}{7.2}{\familydefault}{\mddefault}{\updefault}{\color[rgb]{0,0,0}$\beta$}%
}}}}
\put(3782,-2831){\makebox(0,0)[lb]{\smash{{\SetFigFont{6}{7.2}{\familydefault}{\mddefault}{\updefault}{\color[rgb]{0,0,0}$\beta$}%
}}}}
\put(3146,-2258){\makebox(0,0)[lb]{\smash{{\SetFigFont{6}{7.2}{\familydefault}{\mddefault}{\updefault}{\color[rgb]{0,0,0}$\delta$}%
}}}}
\put(2594,-2338){\makebox(0,0)[lb]{\smash{{\SetFigFont{7}{8.4}{\familydefault}{\mddefault}{\updefault}{\color[rgb]{0,0,0}$\D$}%
}}}}
\end{picture}%

%% file: Caseii1.pspdftex
\begin{picture}(0,0)%
\includegraphics{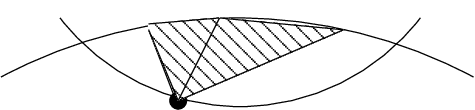}%
\end{picture}%
\setlength{\unitlength}{4144sp}%
\begingroup\makeatletter\ifx\SetFigFont\undefined%
\gdef\SetFigFont#1#2#3#4#5{%
  \reset@font\fontsize{#1}{#2pt}%
  \fontfamily{#3}\fontseries{#4}\fontshape{#5}%
  \selectfont}%
\fi\endgroup%
\begin{picture}(3616,834)(893,-2112)
\put(2161,-1726){\makebox(0,0)[lb]{\smash{{\SetFigFont{12}{14.4}{\familydefault}{\mddefault}{\updefault}{\color[rgb]{0,0,0}$\alpha_1$}%
}}}}
\put(2566,-1816){\makebox(0,0)[lb]{\smash{{\SetFigFont{12}{14.4}{\familydefault}{\mddefault}{\updefault}{\color[rgb]{0,0,0}$\alpha_2$}%
}}}}
\put(4028,-1951){\makebox(0,0)[lb]{\smash{{\SetFigFont{12}{14.4}{\familydefault}{\mddefault}{\updefault}{\color[rgb]{0,0,0}$\ci(R)$}%
}}}}
\put(4156,-1449){\makebox(0,0)[lb]{\smash{{\SetFigFont{12}{14.4}{\familydefault}{\mddefault}{\updefault}{\color[rgb]{0,0,0}${\mathscr C}(v)$}%
}}}}
\end{picture}%

%% file: Caseii2.pspdftex
\begin{picture}(0,0)%
\includegraphics{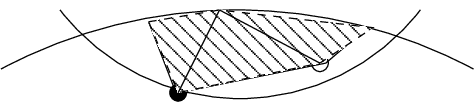}%
\end{picture}%
\setlength{\unitlength}{4144sp}%
\begingroup\makeatletter\ifx\SetFigFont\undefined%
\gdef\SetFigFont#1#2#3#4#5{%
  \reset@font\fontsize{#1}{#2pt}%
  \fontfamily{#3}\fontseries{#4}\fontshape{#5}%
  \selectfont}%
\fi\endgroup%
\begin{picture}(3616,774)(893,-2112)
\put(2161,-1726){\makebox(0,0)[lb]{\smash{{\SetFigFont{12}{14.4}{\familydefault}{\mddefault}{\updefault}{\color[rgb]{0,0,0}$\alpha_1$}%
}}}}
\put(2566,-1816){\makebox(0,0)[lb]{\smash{{\SetFigFont{12}{14.4}{\familydefault}{\mddefault}{\updefault}{\color[rgb]{0,0,0}$\alpha_2$}%
}}}}
\put(3241,-1726){\makebox(0,0)[lb]{\smash{{\SetFigFont{12}{14.4}{\familydefault}{\mddefault}{\updefault}{\color[rgb]{0,0,0}$\gamma$}%
}}}}
\put(4126,-1509){\makebox(0,0)[lb]{\smash{{\SetFigFont{12}{14.4}{\familydefault}{\mddefault}{\updefault}{\color[rgb]{0,0,0}${\mathscr C}(v)$}%
}}}}
\put(4111,-2010){\makebox(0,0)[lb]{\smash{{\SetFigFont{12}{14.4}{\familydefault}{\mddefault}{\updefault}{\color[rgb]{0,0,0}$\ci(R)$}%
}}}}
\end{picture}%

%% file: Caseii3.pspdftex
\begin{picture}(0,0)%
\includegraphics{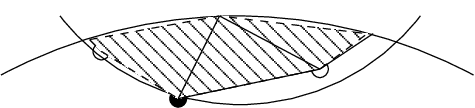}%
\end{picture}%
\setlength{\unitlength}{4144sp}%
\begingroup\makeatletter\ifx\SetFigFont\undefined%
\gdef\SetFigFont#1#2#3#4#5{%
  \reset@font\fontsize{#1}{#2pt}%
  \fontfamily{#3}\fontseries{#4}\fontshape{#5}%
  \selectfont}%
\fi\endgroup%
\begin{picture}(3616,811)(893,-2112)
\put(2026,-1816){\makebox(0,0)[lb]{\smash{{\SetFigFont{12}{14.4}{\familydefault}{\mddefault}{\updefault}{\color[rgb]{0,0,0}$\alpha_1$}%
}}}}
\put(2566,-1816){\makebox(0,0)[lb]{\smash{{\SetFigFont{12}{14.4}{\familydefault}{\mddefault}{\updefault}{\color[rgb]{0,0,0}$\alpha_2$}%
}}}}
\put(3286,-1681){\makebox(0,0)[lb]{\smash{{\SetFigFont{12}{14.4}{\familydefault}{\mddefault}{\updefault}{\color[rgb]{0,0,0}$\gamma$}%
}}}}
\put(4111,-1472){\makebox(0,0)[lb]{\smash{{\SetFigFont{12}{14.4}{\familydefault}{\mddefault}{\updefault}{\color[rgb]{0,0,0}${\mathscr C}(v)$}%
}}}}
\put(4035,-1981){\makebox(0,0)[lb]{\smash{{\SetFigFont{12}{14.4}{\familydefault}{\mddefault}{\updefault}{\color[rgb]{0,0,0}$\ci(R)$}%
}}}}
\end{picture}%

%% file: CaseSec.pspdftex
\begin{picture}(0,0)%
\includegraphics{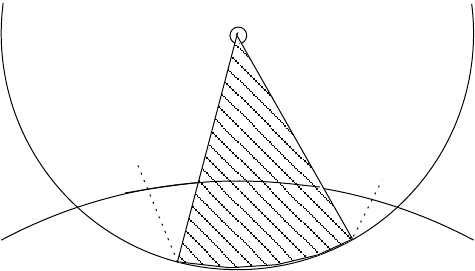}%
\end{picture}%
\setlength{\unitlength}{4144sp}%
\begingroup\makeatletter\ifx\SetFigFont\undefined%
\gdef\SetFigFont#1#2#3#4#5{%
  \reset@font\fontsize{#1}{#2pt}%
  \fontfamily{#3}\fontseries{#4}\fontshape{#5}%
  \selectfont}%
\fi\endgroup%
\begin{picture}(3618,2050)(891,-2095)
\put(4111,-2010){\makebox(0,0)[lb]{\smash{{\SetFigFont{12}{14.4}{\familydefault}{\mddefault}{\updefault}{\color[rgb]{0,0,0}$\ci(R)$}%
}}}}
\put(2663,-772){\makebox(0,0)[lb]{\smash{{\SetFigFont{12}{14.4}{\familydefault}{\mddefault}{\updefault}{\color[rgb]{0,0,0}$2\beta$}%
}}}}
\put(4013,-557){\makebox(0,0)[lb]{\smash{{\SetFigFont{12}{14.4}{\familydefault}{\mddefault}{\updefault}{\color[rgb]{0,0,0}${\mathscr C}(v)$}%
}}}}
\end{picture}%

%% file: BrhoBv.pspdftex
\begin{picture}(0,0)%
\includegraphics{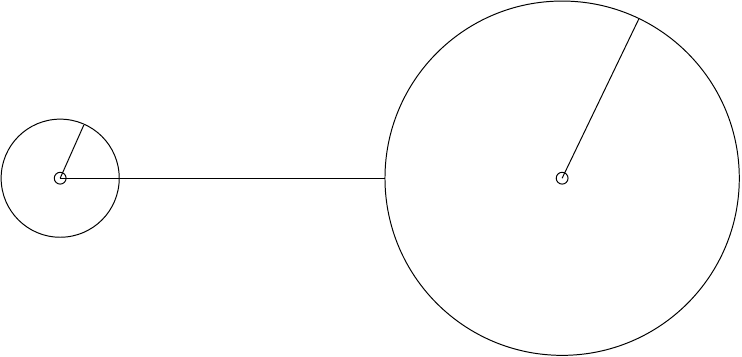}%
\end{picture}%
\setlength{\unitlength}{4144sp}%
\begingroup\makeatletter\ifx\SetFigFont\undefined%
\gdef\SetFigFont#1#2#3#4#5{%
  \reset@font\fontsize{#1}{#2pt}%
  \fontfamily{#3}\fontseries{#4}\fontshape{#5}%
  \selectfont}%
\fi\endgroup%
\begin{picture}(5641,2714)(443,-1868)
\put(1936,-421){\makebox(0,0)[lb]{\smash{{\SetFigFont{14}{16.8}{\familydefault}{\mddefault}{\updefault}{\color[rgb]{0,0,0}$d\geq 100\rho_0/\delta^2$}%
}}}}
\put(766,-286){\makebox(0,0)[lb]{\smash{{\SetFigFont{14}{16.8}{\familydefault}{\mddefault}{\updefault}{\color[rgb]{0,0,0}$\rho_0$}%
}}}}
\put(4051,389){\makebox(0,0)[lb]{\smash{{\SetFigFont{14}{16.8}{\familydefault}{\mddefault}{\updefault}{\color[rgb]{0,0,0}$r(u)\geq 3\delta d$}%
}}}}
\end{picture}%